\documentclass[11pt]{amsart}
\usepackage{stmaryrd}
\usepackage{amssymb}
\usepackage{csvsimple}
\usepackage{siunitx}
\usepackage{hyperref}
\usepackage{url}
\usepackage{graphicx,color}
\usepackage{lineno}           
\usepackage[maxbibnames=99,isbn=false,backend=biber,url=false,eprint=true,style=alphabetic,
            giveninits=true,labelnumber,defernumbers]{biblatex}

\graphicspath{{pics/},{figs/}}
\def\R{\mathbb{R}}

\def\cA{\mathcal{A}}

\def\cF{\mathcal{F}}

\def\cI{\mathcal{I}}

\def\cN{\mathcal{N}}

\def\cP{\mathcal{P}}

\def\cR{\mathcal{R}}

\def\a{\alpha}
\def\b{\beta}
\def\g{\gamma}
\def\d{\delta}

\def\k{\kappa}
\def\l{\lambda}

\def\p{\partial}

\def\veps{\varepsilon}

\def\O{\Omega}
\def\G{\Gamma}

\def\wto{\rightharpoonup}

\def\ou{\overline{u}}
\def\hu{\widehat{u}}
\def\tu{\widetilde{u}}

\newcommand{\dv}[1]{\,{\mathrm d}#1}

\newcommand{\wcheck}[1]{#1\hspace{-.8ex}\mbox{\huge {\lower.45ex \hbox{$\textstyle \check{}$}}} \hspace{.5ex}}

\let\oldmarginpar\marginpar
\renewcommand\marginpar[1]{
  \oldmarginpar[\raggedleft\footnotesize #1]
  {\raggedright\footnotesize #1}}

\newtheorem{definition}{Definition}

\newtheorem{proposition}[definition]{Proposition}

\newtheorem{corollary}[definition]{Corollary}
\newtheorem{remark}[definition]{Remark}
\newtheorem{remarks}[definition]{Remarks}
\newtheorem{example}[definition]{Example}

\newtheorem{algorithm}[definition]{Algorithm}
\newtheorem{assumption}[definition]{Assumption}
\numberwithin{definition}{section}
\definecolor{modmag}{RGB}{179,0,229}

\def\tL{\widetilde{L}}
\def\bc{{\rm bc}}

\begin{document}
\title[Simulation of constrained curves]{Simulation of
constrained elastic curves and application to a conical sheet
indentation problem}
\author{S\"oren Bartels}
\date{\today}
\renewcommand{\subjclassname}{%
\textup{2010} Mathematics Subject Classification}
\begin{abstract}
We consider variational problems that model the bending behavior
of curves that are constrained to belong to given hypersurfaces.
Finite element discretizations of corresponding functionals are 
justified rigorously via
$\G$-convergence. The stability of semi-implicit discretizations
of gradient flows is investigated which provide a practical
method to determine stationary configurations. A particular 
application of the considered models arises in the description 
of conical sheet deformations. 
\end{abstract}
 

\maketitle

\begin{center}
\textsc{Dedicated to the memory of John W. Barrett}  
\end{center} 

\section{Introduction}\label{sec:intro}
The elastic flow of curves has attracted considerable attention  
among applied and numerical analysts within the last decades, cf., e.g., 
\cite{LanSin85,DzKuSc02,DaLiPo14} for analytical results, 
and \cite{DeDzEl05,BGN08,DeDz09,BGN10,BGN11,BGN12,Bart13,PozSti17,
BRR18,BGN19} for results concerning the discretization.
Corresponding applications
occur in the modeling of phase transitions, the description of 
large deformations of elastic rods and ribbons~\cite{AudPom10-book}, and prediction of 
prefered shapes of molecules~\cite{ChGoMa06,CS1}. For the class of inextensible curves, 
which arise naturally as dimensionally reduced descriptions 
in nonlinear elasticity~\cite{antman,MorMul03}, recent
developments concerning the numerical treatment of partial differential
equations with holonomic constraints such as harmonic maps turned out
be useful for their efficient approximation, cf.~\cite{Bart05,BBFP07,Bart16}. 
In this article we 
consider curves that are restricted to belong to given surfaces
and whose behavior is determined by appropriate bending energies. 
To model their relaxation dynamics and find stationary configurations
of low energy we adapt techniques developed in~\cite{Bart13} to 
develop convergent finite element discretizations and stable iterative
numerical schemes. Our approach provides an alternative to the 
methods developed in~\cite{BGN12,BGN19}. Here, motivated by applications
in nonlinear elasticity, we consider curves in 
euclidean space that are parametrized by arclength which allows for
an efficient numerical treatment. Related analytical contributions
are contained in \cite{Linn91,Kois96}.

\bigskip

\subsection{Constrained nonlinear bending} 
We first consider relaxation processes of curves $u$ on a
given surface $S$ whose bending behavior is determined 
by the functional
\[
I[u] = \frac12 \int_0^L |u''|^2 \dv{x}.
\]
Here, we require $u:(0,L)\to \R^3$ to be an arclength parametrized curve, i.e., that
$|u'(x)|=1$ for all $x\in (0,L)$, so that $|u''|^2$ is the squared 
curvature of the curve parametrized by the function~$u$. The constraint
\[
u(x) \in S 
\]
for all $x\in (0,L)$ restricts the curve to belong to the regular hypersurface 
$S\subset \R^3$. We also incorporate boundary conditions modeled
by a bounded and linear 
functional $L_\bc:H^2(\O;\R^3)\to \R^\ell$ and a vector $\ell_\bc\in \R^\ell$.
The setting may describe the behavior of a wire 
on a magnetic surface neglecting effects related to twist. Corresponding
torsion contributions can however be directly included, cf.~\cite{BarRei19-pre}. 
We thus consider the following constrained minimization problem.
\[\label{prob:bend}
\tag{$\textrm{P}_{\textrm{bend}}$} 
\left\{\begin{array}{l}
\text{Find a minimizing curve $u \in H^2(0,L;\R^3)$ for }   \\
\ \displaystyle{ I[u] = \frac12 \int_0^L |u''|^2 \dv{x}}  \\[3mm]
\text{subject to } 
\displaystyle{ u(x)\in S, \, |u'(x)|^2 =1} 
\text{ for all $x\in [0,L]$} \\[3mm]
\text{and } L_\bc[u] = \ell_\bc.
\end{array}\right.
\]
For an initial configuration described by a 
function~$u_0$ and for given boundary conditions, e.g., that the wire is
clamped at one end, the relaxation of the bending energy is modeled
by the formal gradient flow evolution
\[
\p_t u = - I'[u] + (\l u')' + \mu \Phi_S'(u)
\]
for a family of curves $(u(t))_{t\in [0,T]}$ satisfying the 
the initial, holonomic, and boundary conditions 
\[
u(0) = u_0, \quad |u'|^2 = 1, \quad \Phi_S(u) = 0, \quad L_\bc[u] = \ell_\bc.
\]
The functions~$\l$ and~$\mu$ are Lagrange multipliers related to the 
arclength and surface constraints, respectively, where we assume that
the surface~$S$ is given as the zero level set of the function~$\Phi_S$.
With the backward difference quotient operator 
\[
d_t a^k = \frac1\tau(a^k-a^{k-1})
\]
we use a time-stepping scheme that linearizes the constraints at a
previous approximation. By restricting to test functions that belong to
the intersection of the kernels of the linearized constraints
this eliminates the explicit occurence of the
Lagrange multipliers. Since the time-derivative obeys the same 
linear constraints we obtain for an appropriate inner product $(\cdot,\cdot)_*$
and the $L^2$ inner product $(\cdot,\cdot)$
the time-stepping scheme
\[
(d_t u^k,v)_* + ([u^k]'',v'') = 0
\]
subject to the inclusions
\[
d_t u^k, \, v \in \cF[u^{k-1}],
\]
where the set $\cF[u^{k-1}]$ contains the linearized constraints, i.e., 
for a given curve $\hu$ we have
\[
\cF[\hu] = \big\{ v\in H^2(0,L;\R^3): \hu'\cdot v = 0, \, \Phi_S'(\hu)\cdot v = 0, \,
L_\bc[v] = 0 \big\}.
\]
The time-stepping scheme thus requires solving linearly constrained linear
systems of equations, where the constraints are pointwise. We show that the scheme
is unconditionally energy decreasing and that the violation of the
constraints is controlled by the step size independently of the number of
iterations. Our spatial discretization
uses an $H^2$-conforming ansatz and imposes the constraints at the nodes
of a partitioning of the reference interval~$(0,L)$. We justify the
spatial discretization by proving its $\Gamma$-convergence to the continuous
minimization problem. 

\subsection{Geodesic curvature}
An intrinsic variant of the constrained variational problem arises, e.g., in
the description of phase separation processes on surfaces. It replaces
the curvature $\k = |u''|$ by the geodesic curvature $\k_g$. 
For an arclength parametrized curve $u:(0,L)\to S$ it is defined as
\[
\k_g^2 = |u''|^2 - |u''\cdot n_S(u)|^2 =   |u'' \times n_S(u)|^2,
\]
where $n_S= \Phi_S'/|\Phi_S'|$ is a unit normal field on $S$
and where we used that $u''\cdot u' = 0$. 
The corresponding energy functional 
\[
I[u] = \frac12 \int_\O \k_g^2 \dv{s}
\]
still controls the $H^2$ norm of $u$ since the normal part of the curvature
is bounded by the curvature of $S$, i.e., we have
\[
|u''|^2 \le \k_g^2 + c_S^2,
\]
where $c_S$ is the maximum of the principal curvatures of $S$.
This estimate is not availabe when only nodal values of a piecewise
polynomial curve $u_h$ belong to $S$. To cope with this aspect we
introduce a stabilization via a damping parameter $\g\le 1$ in the energy 
functional.
\[\label{prob:geod}
\tag{$\textrm{P}_{\textrm{geod}}^\g$} 
\left\{\begin{array}{l}
\text{Find a minimizing curve $u \in H^2(0,L;\R^3)$ for } \\ 
\ \displaystyle{ I_\g[u] = \frac12 \int_0^L |u''|^2 
- \g |u''\cdot n_S(u)|^2 \dv{x}}  \\[3mm]
\text{subject to } 
\displaystyle{ u(x)\in S, \, |u'(x)|^2 =1} 
\text{ for all $x\in [0,L]$} \\[3mm]
\text{and } L_\bc[u] = \ell_\bc.
\end{array}\right.
\]
We prove that the stabilized problems converge in a variational
sense to the unstabilized original problem as $\g\to 1$.
The stabilization allows us to prove convergence of discretizations.
As an alternative to or in combination with
stabilizations additional constraints may be imposed
to ensure that discrete curves remain sufficiently close to the surface
$S$ so that their second derivative in normal direction is controlled by
the curvature of the surface. This approach however leads to difficulties in 
the iterative solution. For the stabilized problem we follow the ideas
described above with an explicit treatment of the nonlinear term. Hence,
we compute a sequence $(u^k)_{k=0,1,\dots}$ via the recursion
\[\begin{split}
(d_t u^k&,v)_*  + ([u^k]'',v'')  \\
& = \g \big([u^{k-1}]''\cdot n_S(u^{k-1}),
  v ''\cdot n_S(u^{k-1}) + [u^{k-1}]''\cdot n_S'(u^{k-1}) v \big)
\end{split}\]
subject to $d_t u^k,v\in \cF[u^{k-1}]$. Under moderate conditions
on the step size $\tau$ in terms of $\g$ we obtain a monotonicity
property for the iteration. 

\subsection{Conical sheets}
Motivated by the problem of understanding folding and crumpling
deformations of thin elastic sheets, the 
articles~\cite{CerMah05,BKN13,MulOlb14,Olb16,FigMoo18}   
address the situation in which an elastic plate is placed on a circular obstacle 
of radius $r$ and then indented
by an amount $\d$ at the center $C$. The resulting deformation is 
homogeneous along rays starting from the center, points at a distance 
$(r^2+\d^2)^{1/2}$ from the center 
are either in contact with the obstacle or above it. The displacement of these points 
entirely determines the full deformation of the sheet and it therefore
suffices to compute the deformation of the points belonging to this circle. 
The displaced points belong to a sphere and are constrained by the obstacle. 
By an appropriate rescaling we may assume that $r^2 + \d^2 = 1$. A 
cross section of the rotationally symmetric setting through the center $C$ is depicted
in Figure~\ref{fig:sheet_indent_sketch}. The solutions of the two-dimensional
problem and its one-dimensional reduction cannot be rotationally symmetric unless 
the indendation depth $\d$ is trivial. 

\begin{figure}
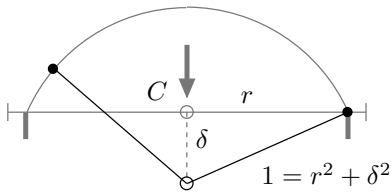
\caption{\label{fig:sheet_indent_sketch} A point $C$ of an initially flat 
elastic sheet (gray line representing cross section) 
is displaced by a distance $\d$. The resulting deformation is
constrained by a circular obstacle at distance $r$ to $C$. Points on the 
deformed sheet (black lines)
at distance $(r^2+\d^2)^{1/2}$ to the center $C$ either touch the obstacle (right 
end point) or are above it (left end point).}
\end{figure}

The corresponding reduced 
description has been rigorously identified in \cite{FigMoo18} and characterizes
the deformation $u:S^1\to \R^3$ of the unit circle $S^1 \subset \R^2$ via a 
minimization of the functional 
\[
I[u] = \frac12 \int_{S^1} \k_g^2 \dv{x} 
\]
in the set of periodic curves $u\in H^2(S^1;\R^3)$
subject to the constraints that $u$ attains its values on the unit sphere
$S=S^2\subset \R^3$ and is inextensible, i.e.,
\[
|u(x)|^2 = 1, \quad |u'(x)|^2 = 1, 
\]
and that the curve does not penetrate the obstacle, i.e., for the vertical
component $u_3$ of $u$ we have 
\[
u_3(x) \ge \d,
\]
for all $x\in S^1$. Because of the unit-length constraints on $u$ and $u'$
we have that the normal curvature $\k_n$ of $u$ is given by
\[
\k_n = u''\cdot u = (u'\cdot u)' - |u'|^2 = -1,
\]
so that for the geodesic part we have
\[
\k_g^2 = \k^2 -\k_n^2 = |u''|^2 -1.
\]
The reduced indentation problem thus leads to the following minimization 
problem for a given indentation depth $\d \ge 0$.
\[\label{prob:indent}
\tag{$\textrm{P}_{\textrm{ind}}$}
\left\{ \begin{array}{l}
\text{Find a minimizing curve $u \in H^2(S^1;\R^3)$ for } \\
\ \displaystyle{ I[u] = \frac12 \int_{S^1} |u''|^2 \dv{s} - \pi}   \\[3mm]
\text{subject to } 
\displaystyle{ |u(x)|^2 = 1, \ |u'(x)|^2 =1, \ u_3(x) \ge \d} 
\text{ for all $x\in S^1$.}
\end{array}\right.
\]
Various features of minimizers have been characterized in~\cite{FigMoo18}, e.g., that
the non-contact zone $\{s\in S^1: u_3(s) > \d\}$ is an interval. Via 
less rigorous arguments it has been stated in~\cite{CerMah05} that minimizers have,
in a certain projection, a unique maximum, i.e., that single folds of the
indented sheet are preferred over double folds, as is observed in reality. 
To investigate such questions via numerical experiments we approximate the 
problem by imposing the inequality constraint using a
penalty approximation, i.e., we consider
\[
I_\veps[u] = \frac12 \int_{S^1} |u''|^2 \dv{x} 
+ \frac{1}{2\veps} \int_{S^1} (u_3-\d)_-^2 \dv{x} - \pi.
\]
The minimimization of $I_\veps$ is done with a gradient flow that linearizes
the constraints and which uses an implicit-explicit treatment of the
penalty term defined via the convex-concave splitting 
\[
(s-\d)_-^2 = (s-\d)^2 - (s-\d)_+^2
\]
i.e., we compute a sequence $(u^k)_{k=0,1,\dots}$ via 
\[
(d_t u^k,v)_* + ([u^k]'',v'') + \veps^{-1} (u_3^k-\d,v_3) 
=  \veps^{-1} ((u_3^{k-1}-\d)_+,v_3)
\]
for all $v\in H^2(S^1;\R^3)$ subject to the linearized unit-length constraints
and periodicity conditions contained in the space $\cF[u^{k-1}]$
\[
d_t u^k,v\in \cF[u^{k-1}].
\]
The resulting iterative method is unconditionally energy monotone and converges to
stationary configurations of low bending energy. 

\subsection{Outline}
The article is organized as follows. In Section~\ref{sec:fe_and_gamma}
we introduce the finite element spaces used to approximate $H^2$ curves
and prove $\G$-convergence results for the model problems. Section~\ref{sec:iterative} is
devoted to the development of stable gradient flow discretizations
used to compute stationary configurations. In Section~\ref{sec:experiments}
we illustrate the theoretical findings by numerical experiments. 


\section{Discretization and $\Gamma$-convergence }\label{sec:fe_and_gamma}

In this section we define suitable finite element spaces to approximate
curves, devise discretizations of the constrained minimization problems,
and prove their variational convergence as discretization parameters 
tend to zero. 

\subsection{Finite element spaces}
We discretize the constrained minimization problems using $H^2$ conforming
finite element spaces for partitions
\[
0 = z_0 < z_1 < \dots <z_J = L
\]
of the interval $(0,L)$ with maximal mesh size  
$h = \max_{j=1,\dots,J} |z_j-z_{j-1}|$ of the subintervals 
$I_j = [z_{j-1},z_j]$. A finite element space subordinated to this
partitioning is defined by imposing continuity and differentiability
of the piecewise cubic curves at the nodes, i.e., we set
\[
V_h = \{v_h \in C^1(0,L;\R^3): v_h|_{I_j} \in \cP_3(I_j)^3, \, j=1,2,\dots,J\},
\]
where $\cP_\ell(I)$ denotes the set of polynomials of maximal degree $\ell\ge 0$
on an interval~$I$.
The degrees of freedom in the space $V_h$ are the function values and
derivatives at the nodes, i.e., 
\[
\big(v_h(z_j),v_h'(z_j)\big)_{j=0,\dots,J}.
\]
Correspondingly, an interpolation operator $\cI_h^{3,1}: H^2(0,L;\R^3)
\to V_h$ is defined by requiring that
\[
\cI_h^{3,1} v(z_j) = v(z_j), \quad [\cI_h^{3,1}v]'(z_j) = v'(z_j)
\]
for $j=0,1,\dots,J$. 
We note that we have the interpolation estimates
\[
\|(\cI_h^{3,1} v - v)^{(k)}\|_{L^p(0,L)} \le c h^{3-k} \|v\|_{W^{3,p}(0,L)}
\]
for all $v\in W^{3,p}(0,L)$  and $k\le 2$,
cf.~\cite{BreSco08-book}.
We also employ the standard piecewise linear interpolation operator
\[
\cI_h: C^0([0,L]) \to W_h
\]
which is defined by requiring
\[
\cI_hw(z_j) = w(z_j)
\]
for $j=0,1,\dots,J$ and thereby defines an element in the space
\[
W_h = \{w_h\in C^0([0,L]): w_h|_{I_j} \in \cP_1(I_j), \, j=1,2,\dots,J\}.
\]
For the interpolation operator we have that
\[
\|\cI_h w - w \|_{L^p(0,L)} \le c h \|w'\|_{L^p(0,L)},
\]
if $p>1$. With the interpolation operator $\cI_h$ we define
discrete inner products and norms via
\[
(v,w)_h = \int_0^L \cI_h [vw] \dv{x}, \quad 
\|v\|_{L^p_h(0,L)}^p = \int_0^L \cI_h [|v|^p] \dv{x}
\]
for $v,w\in C([0,L];\R^\ell)$ and $1\le p\le \infty$, 
where $\|v\|_{L^\infty_h(0,L)} = \max_{j=0,\dots,J} |v(z_h)|$. 

\subsection{Discrete minimization problems}
The pointwise constraints and the nonlinearities require making
certain approxiomations which lead to inconsistency terms. We
impose the arclength condition and the surface constraints at 
the nodes of a partitioning, i.e., we impose that
\[
\cI_h |v_h'|^2 = 1, \quad \cI_h \Phi_S(v_h) = 0,
\]
which is equivalent to the nodal constraints
\[
|v_h'(z_j)| = 1, \quad v_h(z_j) \in S
\]
for $j=0,1,\dots,J$. 
The discrete set of admissible curves is then given by
\[
\cA_h = \big\{ v_h \in V_h : \cI_h|v_h'|^2 = 1, 
\, \cI_h \Phi_S(v_h) = 0, \, L_\bc[v_h] = \ell_\bc \big\}.
\]
It provides an approximation of the continuous set of
admissible curves defined as
\[
\cA = \big\{ v \in H^2(0,L;\R^3): |v'|^2=1, \, \Phi_S(v) =0, \,
L_\bc[v] = \ell_\bc \big\}.
\]
We note that if the continuous admissible set is 
nonempty then also the discrete admissible set is nonempty, i.e.,
we have the implication
\[
v\in \cA \, \implies \, \cI_h^{3,1} v \in \cA_h,
\]
where we assume that $L_\bc[v]$ only depends on the boundary values
of $v$ and $v'$. Our convergence result considers the minimization of
\[
I_\g[u] = 
\begin{cases}
\displaystyle{\frac12 \int_0^L  |u''|^2 - \g |u''\cdot n_S(u)|^2 \dv{x}} & \mbox{for } u \in \cA, \\
+ \infty & \mbox{for } H^2(0,L;\R^3)\setminus \cA,
\end{cases}
\]
with a parameters $\g\in [0,1)$. 
The approximating discrete functionals are 
given by 
\[
I_{\g,h}[u_h] = 
\begin{cases}
\displaystyle{\frac12 \int_0^L  |u_h''|^2 - \g |u_h'' \cdot n_S (u_h)|^2 \dv{x}} & \mbox{for } u_h \in \cA_h, \\
+ \infty & \mbox{for } H^2(0,L;\R^3) \setminus \cA_h,
\end{cases}
\]
for $u_h\in \cA_h$ with the extension by $+\infty$ on
$H^2(0,L;\R^3)\setminus \cA_h$. To prove the convergence $I_{\g,h}\to I_\g$ 
we impose a definiteness property on the boundary condition operator $L_\bc$
and an approximability condition on $\cA$. 

\begin{assumption}[Definiteness]\label{ass:definite}
The seminorm $v\mapsto \|v''\|$ is a norm on the kernel of the operator
$L_\bc:H^2(0,L;\R^3)\to \R^\ell$. 
\end{assumption}

The assumption is satisfied for clamped boundary conditions,
e.g., $L_\bc[v] = (v(0),v'(0))$, and boundary conditions that
fix both end points, i.e., $L_\bc[v] = (v(0),v(L))$. We always assume that the
boundary conditions lead to a nonempty set $\cA$.

\begin{assumption}[Density of smooth curves]\label{ass:density}
The subset of smooth curves $\cA\cap H^3(0,L;\R^3)$ is
dense in $\cA$ with respect to strong convergence in $H^2$. 
\end{assumption}

A relaxation of the assumption
is discussed below in Remark~\ref{rem:relax_constraints}.
The assumption can be justified by regularizing curves in $\cA$,
projecting regular curves on $S$, adjusting the boundary conditions,
and carrying out a suitable reparametrization. We refer the
reader to~\cite{BarRei19-pre} for related ideas.

\begin{proposition}[$\G$-convergence]\label{prop:gamma_genera}
If $0\le \g<1$, $\Phi_S\in C^1(\R^3)$ and 
Assumptions~\ref{ass:definite} and~\ref{ass:density} are
satisfied then we have 
$I_{\g,h}\to I_\g$ in the sense of $\G$-convergence with respect to 
weak convergence in $H^2$, i.e., we have the following: \\
(i) If $(u_h)_{h>0} \subset H^2(0,L;\R^3)$ such that $u_h\in \cA_h$
for every $h>0$ and $I_{\g,h}[u_h]\le c$ then there exists 
$u\in \cA$ such that $u_h \wto u$ in $H^2$ and 
\[
I_\g[u] \le \liminf_{h\to 0} I_{\g,h}[u_h].
\]
(ii) For every $u\in \cA$ there exists a sequence 
$(u_h)_{h>0}\subset H^2(0,L;\R^3)$ such that $u_h \to u$ in
$H^2$ and 
\[
I_\g[u] = \lim_{h\to 0} I_{\g,h}[u_h].
\]
(iii) Weak accumulation points of sequences of quasiminimizers 
$(u_h)_{h>0}$ for the functionals $I_{\g,h}$ in $H^2$ are 
minimizers for $I_\g$.
\end{proposition}

\begin{proof}
(i) If $I_{\g,h}[u_h] \le c$ for a sequence $(u_h)_{h>0}$ then,
since $\g<1$ and since
\begin{equation}\label{eq:repr_proj}
\begin{split}
|u_h''|^2 - \g  |u_h''\cdot & n_S(u_h)|^2  \\
& = (1-\g) |u_h''|^2 + \g \big|\big(I_3 - n_S(u_h)\otimes n_S(u_h)\big) u_h''\big|^2 ,
\end{split}
\end{equation}
we have that the sequence is bounded in $H^2(0,L;\R^3)$ 
and there exists a weak limit $u\in H^2(0,L;\R^3)$ of an 
appropriate subsequence which is not relabeled. 
The boundedness of the linear operator
$L_\bc[v]$ shows that we have $L_\bc[u]=\ell_\bc$.
The compactness of 
the embedding $H^2(0,L) \to W^{1,\infty}(0,L)$ implies
that the sequence $(u_h')_{h>0}$ is strongly convergent
in $L^\infty(0,L;\R^3)$.
Using that $\cI_{\g,h} |u_h'|^2 = 1$ we thus deduce that
\[\begin{split}
\big\||u_h'|^2-1\big\|_{L^2(0,L)} 
&= \big\||u_h'|^2-\cI_h|u_h'|^2\big\|_{L^2(0,L)} \\
&\le 2 ch \|u_h'\|_{L^\infty(0,L)} \| u_h''\|_{L^2(0,L)}^2,
\end{split}\]
which implies that $|u_h'|^2 \to 1$ in $L^2(0,L)$. We have that
\[\begin{split}
\|\Phi_S(u_h)\|_{L^\infty(0,L)} 
&= \|\Phi_S(u_h) - \cI_h \Phi_S(u_h)\|_{L^\infty(0,L)} \\
&\le c h \|\Phi_S'(u_h) u_h'\|_{L^\infty(0,L)}.
\end{split}\]
The pointwise convergence $u_h\to u$ and continuity of $\Phi_S$
imply that $\Phi_S(u)=0$ in $(0,L)$. Hence, we have that $u\in \cA$.
Since 
\[
P_{u_h} = I_3 - n_S(u_h)\otimes n_S(u_h) \to P_u = I_3 - n_S(u) \times n_S(u)
\]
strongly in $L^\infty(0,L;\R^{3\times 3})$ it follows that
$P_{u_h} u_h'' \wto P_u u''$ in $L^2(0,L;\R^3)$ and the 
weak lower semicontinuity of the $L^2$ norm in combination 
with the identity~\eqref{eq:repr_proj} shows that
\[
\int_0^L |u''|^2 - \g |u''\cdot n_S(u)|^2 \dv{x}
\le \liminf_{h\to 0} \int_0^L |u_h''|^2 - \g |u_h''\cdot n_S(u_h)|^2 \dv{x},
\]
i.e., that $I_\g[u]\le \liminf_{h\to 0} I_{\g,h}[u_h]$.  \\
(ii) Since $I_\g$ is continuous on $\cA$ with respect to 
strong convergence in $H^2$ and because of Assumption~\ref{ass:density}, we may 
assume that $u\in \cA \cap H^3(0,L;\R^3)$. Letting
$u_h = \cI_h^{3,1}u$ we have that $u_h \in \cA_h$,
$u_h \to u$ in $H^2$, and $I_\g[u]= \lim_{h\to 0} I_{\g,h}[u_h]$.  \\
(iii) The convergence of quasi-minimizers is an immediate
consequence of the equicoercivity of the functionals $I_{\g,h}$ 
owing to the condition $\g<1$ and assertions~(i) and~(ii).
\end{proof}

\begin{remark}\label{rem:relax_constraints}
To avoid Assumption~\ref{ass:density} one may impose the
arclength and surface constraints in a relaxed sense
in defining $\cA_h$, i.e., using
\[\begin{split}
\widetilde{\cA}_h & = \big\{ v_h \in V_h : , \, L_\bc[v_h] = \ell_\bc, \\
& \qquad \|\cI_h|v_h'|^2 - 1\|_{L^\infty(0,L)} \le \a_h, 
\, \| \cI_h \Phi_S(v_h)\|_{L^\infty(0,L)} \le \b_h \big\},
\end{split}\]
with $h$-dependent parameters $\a_h,\b_h>0$. In this case,
one may construct a recovery sequence $u_h$ in part (ii) 
of the Proposition by letting $\tu\in C^\infty(0,L;\R^3)$ 
be a regularization of $u\in \cA$ which obeys the boundary
conditions and define $u_h = \cI_h^{3,1} u$. If 
$\a_h,\b_h$ are appropriately chosen we have 
$u_h\in\widetilde{\cA}_h$ and $u_h \to u$ in $H^2$.
\end{remark}

\subsection{Application to model problems}
We next apply the abstract $\G$-convergence result to the model
problems defined by the variational problems~\eqref{prob:bend}, 
\eqref{prob:geod}, and~\eqref{prob:indent}. We assume throughout
the following that $\Phi_S\in C^1(\R^3)$ and that
Assumptions~\ref{ass:definite} and~\ref{ass:density} are satisfied
and always consider weak convergence in $H^2$.
The discretization of the constrained nonlinear bending 
problem~\eqref{prob:bend} is defined as:
\[\label{prob:bend_h}
\tag{$\textrm{P}_{\textrm{bend}}^h$} 
\left\{\begin{array}{l}
\text{Find a minimizing curve $u_h\in \cA_h$ for } \\
\ \displaystyle{I[u_h] = \frac12 \int_0^L |u_h''|^2 \dv{x}.}
\end{array}\right.
\] 
A convergence result is obtained from choosing $\g=0$ in
Proposition~\ref{prop:gamma_genera}.

\begin{corollary}[Constrained nonlinear bending]
The minimization problems~\eqref{prob:bend_h} approximate the
problem~\eqref{prob:bend} as $h\to 0$.
\end{corollary}

A discretization of the geodesic curvature minimization 
problem~\eqref{prob:geod} is defined as:
\[\label{prob:geod_h}
\tag{$\textrm{P}_{\textrm{geod}}^{\g,h}$} 
\left\{\begin{array}{l}
\text{Find a minimizing curve $u_h\in \cA_h$ for } \\
\ \displaystyle{I_{\g,h} [u_h] = \frac12 \int_0^L |u_h''|^2 - \g |u_h'' \cdot n_S(u_h)|^2 \dv{x}.}
\end{array}\right.
\] 
This problem approximates for fixed $0<\g<1$ the stabilized
problem~\eqref{prob:geod} which is a direct consequence of 
Proposition~\ref{prop:gamma_genera}. We also have that the
regularized minimization problems converge for $\g\to 1$ 
to the original, unstabilized problem defined with $\g=1$.

\begin{corollary}[Geodesic curvature minimization]
The minimization problems~\eqref{prob:geod_h} approximate
problem~\eqref{prob:geod} as $h\to 0$. For $\g\to 1$ 
problems~\eqref{prob:geod} approximate problem~\eqref{prob:geod}
with $\g=1$.
\end{corollary}

\begin{proof}
The first part follows from Proposition~\ref{prop:gamma_genera}.
To prove the second part one uses that second derivatives of
arclength-parametrized curves on $S$ are bounded by their geodesic
curvature. 
\end{proof}

\begin{remarks}
(i) For an efficient numerical realization it is helpful to replace 
the function $u_h''\cdot n_S(u_h)$ by $u_h''\cdot n_S(\ou_h)$,
where $\ou_h$ is a piecewise constant approximation 
of $u_h$. The approximation result remains valid if
$u_h - \ou_h \to 0$ in $L^\infty(0,L;\R^3)$ for every 
bounded seqence $(u_h)_{h>0}$ in $H^2(0,L;\R^3)$, e.g., if
$\ou_h$ is defined via the midpoint values of $u_h$. \\
(ii) A modification of the method is necessary to justify
a joint limit passage $(h,\g) \to (0,1)$. In particular, control
on the normal part of $u_h''$ is needed, e.g., via requiring
that $u_h'(z_j)$ is a tangent vector at every node $z_j$, 
$j=0,1,\dots,J$. 
\end{remarks}

A discretization of the sheet indentation 
problem~\eqref{prob:indent} is defined as:
\[\label{prob:indent_h}
\tag{$\textrm{P}_{\textrm{ind}}^{h,\veps}$} 
\left\{\begin{array}{l}
\text{Find a minimizing curve $u_h\in \cA_h$ for } \\
\ \displaystyle{I_{h,\veps}[u_h] = \frac12 \int_0^L |u_h''|^2 
+ \frac1{2\veps} \int_0^L \cI_h (u_{3,h}-\d)_-^2\dv{x}.}
\end{array}\right.
\] 
A convergence result is obtained from choosing $\g=0$ in
Proposition~\ref{prop:gamma_genera} and showing that the
penalty term turns into a rigid constraint as $(h,\veps)\to 0$. 

\begin{corollary}[Constrained nonlinear bending]
Assume that Assumption~\ref{ass:density} holds with $\cA$ 
replaced by the set of functions $u\in \cA$ with $u_3 \ge \d$.
Then the minimization problems~\eqref{prob:indent_h} approximate the
problem~\eqref{prob:indent} as $(h,\veps)\to 0$.
\end{corollary}

\begin{proof}
Certain modifications of the proof of 
Proposition~\ref{prop:gamma_genera} are required. If the sequence
$(u_h)_{h>0}$ is such that $I_{h,\veps}[u_h] \le c$ then
we have $\|(u_{3,h}-\d)_-\|_{L^2_h(0,L)}^2 \le 2 c \veps$ and every
weak accumulation point $u\in H^2(0,L;\R^3)$ satisfies $u_3\ge \d$.
Since the penalty term is nonnegative we have that 
$\liminf_{(h,\veps)\to 0} I_{h,\veps}[u_h] \ge I[u]$. For 
a curve $u\in \cA\cap C^\infty(0,L;\R^3)$ obeying
the constraint $u_3\ge \d$ we have that the interpolants 
$u_h = \cI_h^{3,1}u$ also satisfy $\cI_h u_{3,h} \ge \d$ so that
the penalty term in the functional disappears and 
the second part of the proof of Proposition~\ref{prop:gamma_genera}
applies verbatimly. 
\end{proof}


\section{Discrete gradient flows on surfaces}\label{sec:iterative}
We investigate in this section the stability of gradient flows for 
curvature energies defined on classes of arclength parametrized
curves that belong to a given 
surface. The first model uses the full bending energy, the second 
one is defined by the geodesic curvature, while the third problem 
involves an obstacle constraint. 

\subsection{Constrained elastic flow of curves}
Minimizing the bending energy of curves restricted to a surface
$S$ subject to inextensibility and boundary conditions as 
formulated in problem~\eqref{prob:bend} leads to 
gradient flows such as 
\[
\p_t u = - u^{(4)} + (\l u')' + \mu \Phi_S' (u),
\]
where $\l$ and $\mu$ are Lagrange multipliers related to 
inextensibility and surface constraints. More generally, 
given a metric $(\cdot,\cdot)_*$ defined on $L^2(0,L;\R^3)$
we consider the evolution problem
\[
(\p_t u,v)_* + (u'',v'') = 0
\]
that determines a family  $u:[0,T] \to H^2(0,L;\R^3)$ of 
curves satisfying 
\[
u(0)=u_0, \quad u(t)\in \cA
\]
for all $t\in [0,T]$. We require the test functions 
$v\in H^2(0,L;\R^3)$ to belong to the linearization of
$\cA$ at $u(t)$, i.e., that $v\in \cF[u(t)]$, where
\[
\cF[\hu] = \big\{
v\in H^2(0,L;\R^3): \Phi_S'(\hu) \cdot v = 0, \quad 
\hu'\cdot v = 0, \quad L_\bc[v]= 0
\big\}.
\]
Note that also $\p_t u(t) \in \cF[u(t)]$. To discretize
the evolution equation we use a step size $\tau>0$ and
the backward difference operator 
\[
d_t u^k = \frac{1}{\tau} (u^k-u^{k-1}).
\]
For a partition $z_0<z_1<\dots<z_J$ of $(0,L)$ we define
the discrete linearized admissible space 
\[
\cF_h [\hu_h] = \big\{ v_h \in V_h: \cI_h[\Phi_S'(\hu_h) \cdot v_h] = 0, \
\cI_h [\hu_h' \cdot v_h'] = 0, \ L_\bc[v_h] = 0 \big\},
\] 
i.e., the orthogonality relations are imposed only
at the nodes $z_0,z_z,\dots,z_J$, in accordance with the
definition of the discrete admissible set $\cA_h$. This leads to the
following algorithm. 

\begin{algorithm}[Constrained curvature flow]\label{alg:constr_curv}
Choose $u_h^0\in V_h$ such that $\cI_h \Phi_S(u_h^0) = 0$ and 
$\cI_h |[u^0]'|^2 =1$ and $L_\bc[u_h^0] = \ell_\bc$. Set $k=0$. \\
(1) Compute $d_t u_h^k \in V_h$ such that 
\[
(d_t u_h^k, v_h)_* + ([u_h^{k-1}+\tau d_t u_h^k]'',v_h'') = 0
\]
for all $v_h \in V_h$ subject to the constraints
\[
d_t u_h^k, \, v_h \in \cF_h[u_h^{k-1}].
\]
(2) Define $u_h^k = u_h^{k-1} + \tau d_t u_h^k$; set $k\to k+1$, and
continue with~(1). 
\end{algorithm}

The iteration of Algorithm~\ref{alg:constr_curv} is unconditionally 
well defined and energy
decreasing and leads to a violation of the constraints that is controlled
by the step size $\tau>0$.

\begin{proposition}\label{prop:constr_curv}
(i) Algorithm~\ref{alg:constr_curv} defines a sequence 
$(u_h^k)_{k=0,1,...} \subset V_h$ such that for every $K\ge 0$ we have
\[
I[u_h^K] + \tau \sum_{k=1}^K \|d_t u_h^k\|_*^2  \le I[u_h^0] = e_{0,h}.
\]
(ii) Assume that $u_h^0\in \cA_h$ and
that the inner product $(\cdot,\cdot)_*$ induces a norm $\|\cdot\|_*$ with
\[
\|v_h'\|_{L^\infty_h(0,L)}^2 = \|\cI_h v_h'\|_{L^\infty(0,L)}^2 \le c_* \|v_h\|_*^2
\]
for all $v_h \in V_h$ and $|\Phi_S''(s)| \le c_{S,2}(1+ |s|^r)$ for all 
$s\in \R^3$. Then, we have for every $K \ge 0$ that
\[
\max_{k=0,1,\dots,K}
\||[u_h^k]'|^2-1\|_{L_h^\infty(0,L)} \le c_* \tau e_{0,h},
\]
and
\[
\max_{k=0,1,\dots,K} \|\Phi_S(u_h^k)\|_{L_h^\infty(0,L)} 
\le c_* c_{S,2} c' \tau e_{0,h}^{r+1}.
\]
\end{proposition}

\begin{proof}
We test the formulation of Step~(1) of Algorithm~\ref{alg:constr_curv} 
with $v_h = d_t u_h^k$ to deduce with a binomial formula that
\[
\|d_t u_h^k\|_*^2 + d_t \frac12 \|[u_h^k]''\|^2 
+ \frac{\tau}{2} \|[d_t u_h^k]''\|^2 = 0.
\]
A summation over $k=1,2,\dots,K$ yields the asserted energy estimate.
The nodewise orthogonality $[d_t u_h^k]' \cdot [u_h^{k-1}]' = 0$ and the
relation $u_h^k = u_h^{k-1}+\tau d_t u_h^k$ imply that at
every node $z\in \cN_h$ we have 
\[
|[u_h^k]'|^2 = |[u_h^{k-1}]'|^2 + \tau^2 |[d_t u_h^k]'|^2 
= \dots = 1 + \tau^2 \sum_{\ell = 1}^k |[d_t u_h^k]'|^2.
\]
The energy bound and the assumed inequality for $\|\cdot\|_*$ imply
the bound for the arclength violation. For the surface constraint we
note that the application of a Taylor formula and the fact that 
$d_t u_h^k \in \cF_h[u_h^{k-1}]$  yield that at every node we have 
\[
\Phi_S(u_h^k) = 
\Phi_S(u_h^{k-1}) + \frac12 \tau^2 \Phi_S''(\xi_h^k) [d_t u_h^k,d_t u_h^k].
\]
Repeating this argument and using $\Phi_S(u_h^0)=0$ at the nodes 
we infer with the assumed estimate for $\Phi_S''$ that 
\[
\|\cI_h \Phi_S(u_h^k)\|_{L^\infty(0,L)} \le 
\frac12 \tau c_{S,2} \big(1+\|\cI_h \xi_h^k\|_{L^\infty(0,L)}^r\big) 
\tau \sum_{\ell=1}^k \|\cI_h d_t u_h^\ell\|_{L^\infty(0,L)}^2.
\]
Since the nodal values of 
$\xi_h^k$ belongs to the line segment connecting $u_h^k$ 
and $u_h^{k-1}$ we may incorporate the discrete $L^\infty$ estimates
to deduce the estimate for the nodewise surface 
constraint violation. 
\end{proof}

\subsection{Geodesic curvature flow}
To develop an iterative scheme for the approximate solution of the 
geodesic curvature problem~\eqref{prob:geod} we follow the ideas 
used for the constrained bending problem and use that
\[
\k_g^2 = |u''|^2 - |u''\cdot n_S(u)|^2.
\]
To control the nonlinear second term by the first one, we introduce
a stabilization via a damping factor $\g_\veps = (1-\veps^2)$. This leads
to the functional
\[
I_\veps[u] = \frac12 \int_0^L |u''|^2 - \g_\veps |u''\cdot n_S(u)|^2 \dv{x}.
\]
Because of the stabilization we have the implication
\[
I_\veps[u] \le c_0 \quad \implies \quad \|u''\|^2 \le 2 c_0 \veps^{-2}.
\]
While on the continuous level the geodesic curvature of 
a curve on the surface $S$ controls the full curvature this 
is not the case for the discretization and hence necessitates the
stabilization. We assume that
\[
n_S: \R^3 \to \R^3
\]
is a $C^2$ vector field which coincides with the normal field
on $S$, i.e., we have
$n_S|_S = \frac{\Phi_S'(u)}{|\Phi_S'(u)|}$. We further
assume that $n_S$ has bounded derivatives. 
To simplify notation we use the mapping
\[
G_\veps[u] = \frac{\g_\veps}{2} \int_0^L |u''\cdot n_S(u)|^2 \dv{x}.
\]
The constrained gradient flow for $I_\veps$ can thus be 
represented as
\[
(\p_t u,v)_* + (u'',v'') =  G_\veps'[u;v],
\]
where 
\[
G_\veps'[u;v] = \g_\veps \int_0^L u'' \cdot n_S(u) 
\big(v'' \cdot n_S(u) + u''\cdot n_S'(u)v \big)\dv{x}.
\]
We note that we have
\[
G_\veps'[u;v] \le \g_\veps \big(\|u''\|\|v''\| 
+ c_{n_S} \|u''\|^2 \|v\|_{L^\infty(0,L)}\big).
\]
For ease of presentation we consider a semi-discrete setting. All
arguments carry over to the case of a spatially discrete scheme.  

\begin{algorithm}[Constrained geodesic curvature flow]\label{alg:geod}
Choose $u^0\in V$ such that $\Phi_S(u^0) = 0$ and 
$|[u^0]'|^2 =1$ and $L_\bc[u^0] = \ell_\bc$. Set $k=0$. \\
(1) Compute $d_t u^k \in V$ such that
\[
(d_t u^k, v)_* + ([u^{k-1}+\tau d_t u^k]'',v'') = G_\veps'[u^{k-1};v]
\]
for all $v \in V$ subject to the constraints
\[
d_t u^k, \, v \in \cF[u^{k-1}].
\]
(2) Define $u^k = u^{k-1} + \tau d_t u^k$; set $k\to k+1$, and
continue with~(1). 
\end{algorithm}

We have the following stability properties for Algorihm~\ref{alg:geod}.

\begin{proposition}\label{prop:stab_geod} 
Assume that there exists $c_* >0$ such that 
\[
\|v\|_{L^\infty(0,L)} + \|v''\| \le c_* \|v\|_*
\]
for all $v\in V$. \\
(i) There exists $c_3 \ge 0$ such that if $c_3 \tau \veps^{-1} \le 1/2$ then
the iterates of Algorithm~\ref{alg:geod} satisfy for all $K\ge 0$ 
\[
I_\veps[u^K] + (1-c_3 \tau \veps^{-1}) \tau \sum_{k=1}^K \|d_t u^k\|_*^2 
\le I_\veps[u^0].
\]
(ii) Under the above condition the bounds on the constraint violation errors 
apply as in Proposition~\ref{prop:constr_curv}~(ii).
\end{proposition}

\begin{proof}
We argue by induction and assume that the energy estimate and
the constraint violation bounds have 
been established up to some number $k-1\ge 0$ so that 
\[
I_\veps[u^{k-1}] + \frac\tau2 \sum_{\ell=1}^{k-1} \|d_t u^\ell\|_*^2 \le I_\veps[u^0] =e_0.
\]
This implies that
\[
\|[u^{k-1}]''\|\le \sqrt{2} e_0^{1/2} \veps^{-1}.
\]
By the assumption on the boundary data we thus have
that $\|u^{k-1}\|_{H^2(0,L)} \le c_1 \veps^{-1}$. 
Moroever, we have that $\|u^{k-1}\|_{L^\infty(0,L)} \le c$. 
To derive an auxiliary bound we choose $v=d_t u^k$ in Step~(1) of 
Algorithm~\ref{alg:geod}. Incorporating the bound for $G_\veps'$
and noting $\g_\veps \le 1$ this leads to 
\[\begin{split}
\|d_t u^k\|_*^2 & + d_t \frac12 \|[u^k]''\|^2 + \frac{\tau}{2} \|[d_t u^k]''\|^2 \\
& \le \|[u^{k-1}]''\| \|[d_t u^k]''\|
+ \|[u^{k-1}]''\|^2 c_{n_S}  \|d_t u^k\|_{L^\infty(0,L)}.
\end{split}\]
By the assumption on the inner product $(\cdot,\cdot)_*$  we have that
\[
\|d_t u^k\|_{L^\infty(0,L)} + \|[d_t u^k]''\| \le c_* \|d_t u^k\|_*
\]
and we deduce that
\[
\frac12 \|d_t u^k\|_*^2 + d_t \frac12 \|[u^k]''\|^2 
\le c_1 \veps^{-2}. 
\]
Hence, by choosing $\tau$ sufficiently small, we have that
\[
\|[u^k]''\|^2 \le \|[u^{k-1}]''\|^2 + 2\tau c_1 \veps^{-2} \le 5 e_0 \veps^{-2}.
\]
We next improve the latter bound by choosing again $v=d_t u^k$ and
using
\[
G_\veps [u^k] - G_\veps[u^{k-1}] = \tau G_\veps'[u^{k-1};d_t u^k] 
+ \tau^2 G_\veps''[\xi^k;d_t u^k,d_t u^k],
\]
where $G_\veps''[\xi^k;d_t u^k,d_t u^k]$ is a formal representation of the
Taylor remainder term 
\[
\cR_{G_\veps}[u^{k-1},u^k;d_t u^k,d_t u^k]
= \int_0^1 (1-s)G_\veps''[u^{k-1}+s(u^k-u^{k-1});d_t u^k,d_t u^k] \dv{s}.
\]
With the bounds for $u^k$ and $u^{k-1}$ we obtain that
\[
\big|G_\veps''[\xi^k;d_t u^k,d_t u^k]\big| \le c_2  (1+\|[\xi^k]''\|) \|[d_t u^k]''\|^2
\le c_2' \veps^{-1} \|[d_t u^k]''\|^2.
\]
We thus obtain that 
\[\begin{split}
\|d_t u^k\|_*^2 + d_t & \frac12 \|[u^k]''\|^2 + \frac12 \|[d_t u^k]''\|^2
= G_\veps'[u^{k-1};d_t u^k] \\
&= d_t G_\veps[u^k] - \tau G_\veps''[\xi^k;d_t u^k,d_t u^k] 
\le d_t G_\veps[u^k] + \tau  c_2'  \veps^{-1} \|d_t u^k\|_*^2.
\end{split}\]
This proves the energy monotonicity and hence part~(i) of the proposition.
Part~(ii) follows as in the proof of Proposition~\ref{prop:constr_curv}.
\end{proof}

\subsection{Conical sheet indentation flow}
To iteratively solve the reduced conical sheet indentation problem~\eqref{prob:indent}
we include the obstacle condition $u_s(x)\ge \d$ via a penalty
term in the energy functional, i.e., 
\[
I_\veps[u] = \frac12 \int_{S^1} |u''|^2 \dv{x} 
+ \frac{1}{2\veps} \int_{S^1} (u_3-\d)_-^2 \dv{x},
\]
where $(s)_-=\min\{s,0\}$. 
The discretization of the related gradient flow
\[
(\p_t u,v)_* + (u'',v'') + \veps^{-1} ((u_3-\d)_-,v_3) = 0
\]
uses the convex-concave splitting 
\[
(u_3-\d)_-^2 = (u_3-\d)^2 - (u_3-\d)_+^2
\]
and an implicit treatment of the corresponding monotone and
an explicit treatment of the corresponding antimonotone terms,
i.e., we use the time-stepping scheme
\[
(d_t u^k,v)_* + ([u^k]'',v'') + \veps^{-1} (u_3^k-\d,v_3) 
= \veps^{-1} \big((u_3^{k-1}-\d)_+,v_3\big).
\]
A spatial discretization leads to the following algorithm
where periodicity is guaranteed via an appropriate definition
of the operator $L_\bc$. 

\begin{algorithm}[Conical sheet flow]\label{alg:conical_sheet}
Choose $u_h^0\in V_h$ such that $\cI_h \Phi_S(u_h^0) = 0$ and 
$\cI_h |[u_h^0]'|^2 =1$ and $\cI_h u_h^0 \ge \d$ 
and $L_\bc[u_h^0] = \ell_\bc$. Set $k=0$. \\
(1) Compute $d_t u_h^k \in V_h$ such that 
\[\begin{split}
(d_t u_h^k, v_h)_* + ([u_h^{k-1}+\tau d_t u_h^k]'',v_h'') 
+ \veps^{-1} &  (u_{h,3}^k-\d,v_{h,3})_h \\
& = \veps^{-1} \big((u_{h,3}^{k-1}-\d)_+,v_{h,3}\big)_h
\end{split}\]
for all $v_h \in V_h$ subject to the constraints
\[
d_t u_h^k, \, v_h \in \cF_h[u_h^{k-1}].
\]
(2) Define $u_h^k = u_h^{k-1} + \tau d_t u_h^k$; set $k\to k+1$, and
continue with~(1). 
\end{algorithm}

The iteration of Algorithm~\ref{alg:conical_sheet} 
has the same features as that of Algorithm~\ref{alg:constr_curv}. We use
the discrete penalized energy functional
\[
I_{h,\veps}[u_h] = \frac12 \int_{S^1} |u_h''|^2 \dv{x}
+ \frac{1}{2\veps} \int_{S^1} \cI_h (u_{h,3} -\d)_-^2 \dv{x}.
\]

\begin{proposition}\label{prop:sheet_iter}
(i) Assume that $u_h^0\in \cA_h$ with $u_{h,3}^0 \ge \d$. 
Algorithm~\ref{alg:conical_sheet} defines a sequence 
$(u_h^k)_{k=0,1,...} \subset V_h$ such that for every $K\ge 0$ we have
\[
I_{h,\veps}[u_h^K] + \tau \sum_{k=1}^K \|d_t u_h^k\|_*^2  
\le I_{h,\veps}[u_h^0] = e_{0,h}.
\]
(ii) Under the above conditions the bounds on the constraint violation errors 
apply as in Proposition~\ref{prop:constr_curv}~(ii).
\end{proposition}

\begin{proof}
We follow the steps of the proof of Proposition~\ref{prop:constr_curv}
and use $v_h = d_t u_h^k$ in Step~(1) of 
Algorithm~\ref{alg:conical_sheet}. Defining the convex
and concave functions $p_{cx}$ and $p_{cv}$, suitably embedded
into $\R^3$, via 
\[
p_{cx}(s) = (s_3-\d)^2 e_3, \quad p_{cv}(s) = -(s_3-\d)_+^2 e_3,
\]
with the canonical basis vector $e_3\in \R^3$, we thus have
\[\begin{split}
\|d_t u^k\|_*^2 + d_t & \frac12 \|[u_h^k]''\|^2 + \frac{\tau}{2} \|[d_t u_h^k]''\|^2 \\
& = - \veps^{-1} (p_{cx}'(u_h^k),d_t u_h^k) - \veps^{-1} (p_{cv}'(u_h^{k-1}),d_tu_h^k).
\end{split}\]
The convexity of $p_{cx}$ and $-p_{cv}$ imply that we have
\[\begin{split}
p_{cx}'(u_h^k) \cdot (u_h^{k-1}-u_h^k) + p_{cx}(u_h^k) &\le p_{cx}(u_h^{k-1}), \\ 
-p_{cv}'(u_h^{k-1}) \cdot (u_h^k-u_h^{k-1}) - p_{cv}(u_h^{k-1}) &\le -p_{cv}(u_h^k).
\end{split}\]
By adding the inequalities and dividing by $\tau$ we find that 
\[
- \big(p_{cx}'(u_h^k) + p_{cv}'(u_h^{k-1})\big) \cdot d_t u_h^k 
\le - d_t \big( p_{cx}(u_h^k) + p_{cv}(u_h^k)\big).
\]
Combining the estimates implies the unconditional energy decay
property. The remaining part (ii) is derived as in the proof 
of Proposition~\ref{prop:constr_curv}.
\end{proof}

\begin{remark}
To obtain a consistency property for the discrete gradient flow
as an approximation of a corresponding continuous gradient flow
a condition relating the step-size $\tau$ and the penalty parameter
$\veps$ is required.
\end{remark}

\section{Numerical experiments}\label{sec:experiments}
We illustrate the performance of the numerical methods devised in the previous
sections by various numerical experiments which are specified in the following
subsections. The implementation of the algorithms was realized in \textsc{Matlab} 
with a direct solution of the linear systems of equations. The evolution
metric $(\cdot,\cdot)_*$ was always chosen to coincide with $L^2$ inner
product which leads to a mesh-dependent constant $c_*$ in 
Propositions~\ref{prop:constr_curv}, \ref{prop:stab_geod}, 
and~\ref{prop:sheet_iter}. We observe however good stability properties
for the resulting discrete $L^2$ flow. 

\subsection{Elastic and geodesic flows on a torus}
We compare discrete relaxation dynamics for curves on a
torus that are determined  by the
elastic bending energy and by the geodesic curvature functional.
The torus $T_{r,R}$ has radii $R=2$ and $r=1$ and is described
by the zero level set of the function 
\[
\Phi_S(s) = (|s|^2 + R^2 - r^2)^2-4R^2(|s|^2-s_3^2).
\]
The following example defines an open curve on $T_{r,R}$. 

\begin{example}
Let $\tL = 2\pi$ and for $x\in (0,\tL)$ define $\tu^0:(0,\tL)\to T_{r,R}$ via
\[
\tu^0(x) = 
\begin{bmatrix}
\sin(ax)\big(R+\sin(bx)r\big) \\ \cos(ax) \big(R+\sin(bx)r\big) \\ \cos(bx) r
\end{bmatrix}
\]
The curve $u^0:(0,L) \to T_{r,R}$ is obtained from a re-parametrization 
of $\tu^0$. 
\end{example}

We use clamped boundary conditions at $x=0$ that fix the initial
position and tangent, i.e., we have
\[
L_\bc[u] = (u(0),u'(0)).
\]
For a partition of the interval $(0,L)$ we ran Algorithms~\ref{alg:constr_curv}
and~\ref{alg:geod} with the parameters
\[
J = 80, \quad h = 2\pi /J,  \quad \tau = h, \quad \g = 1-h.
\]
Figure~\ref{fig:snaps_torus_comp} shows snapshots of the iterations.
We observe that the curve changes quicker initially in the case of
the bending energy and slightly slower for the geodesic curvature
functional. This behavior is also seen in the energy plot shown 
in Figure~\ref{fig:torus_comp_eners} where we plotted the energies
in dependence of the iteration numbers. In Figure~\ref{fig:torus_instability}
we illustrate for the evolution of closed curves on the torus
the necessity of a stabilizing damping factor 
for the geodesic curvature flow. When no stabilization is used,
i.e., in case that $\g=1$, then energy monotonicity fails and the 
discrete curves fail to belong to a small neighborhood of the
given surface. 

\begin{figure}[p]
\includegraphics[width=.52\linewidth]{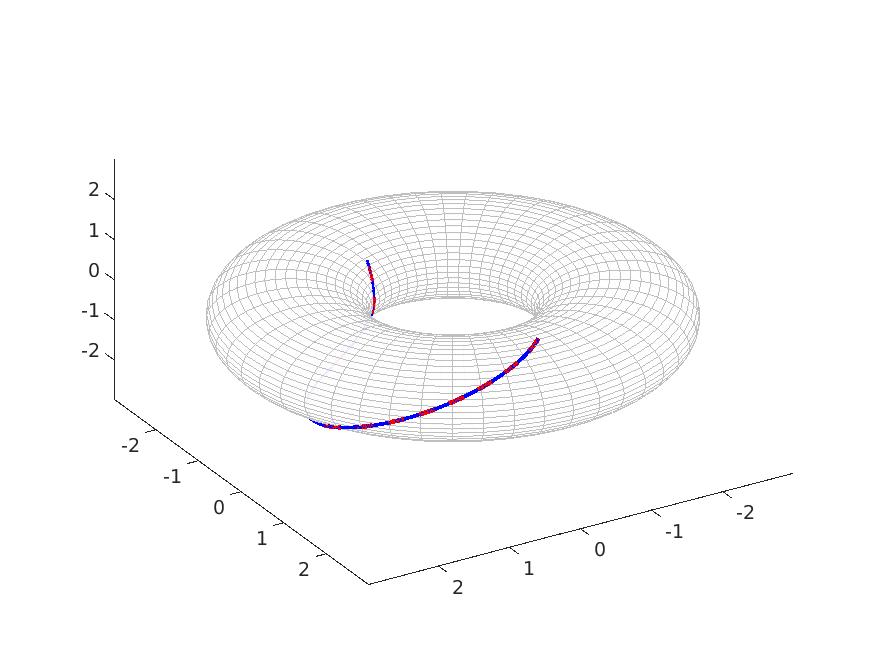}\hspace*{-8mm}
\includegraphics[width=.52\linewidth]{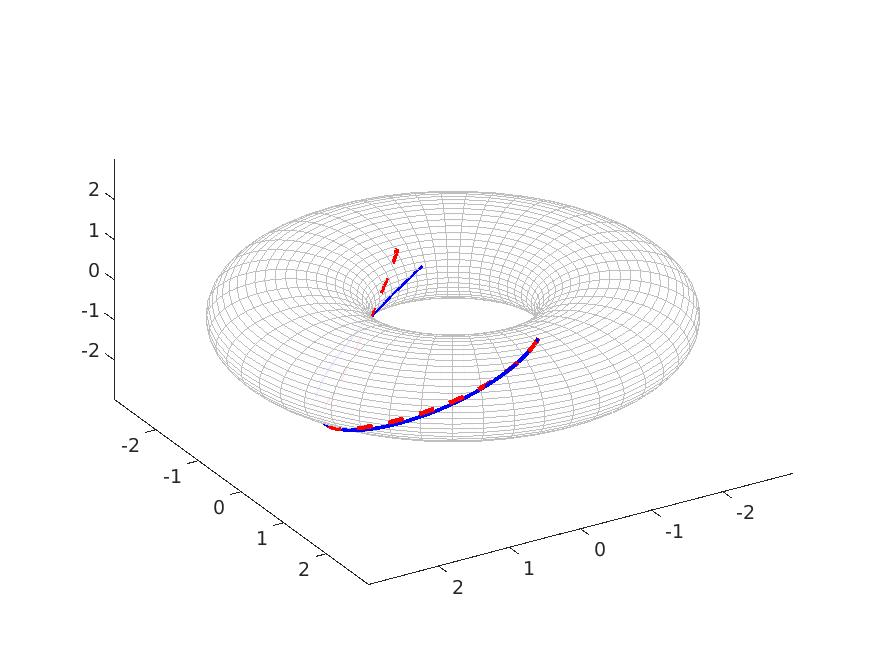} \\[-5.1mm]
\includegraphics[width=.52\linewidth]{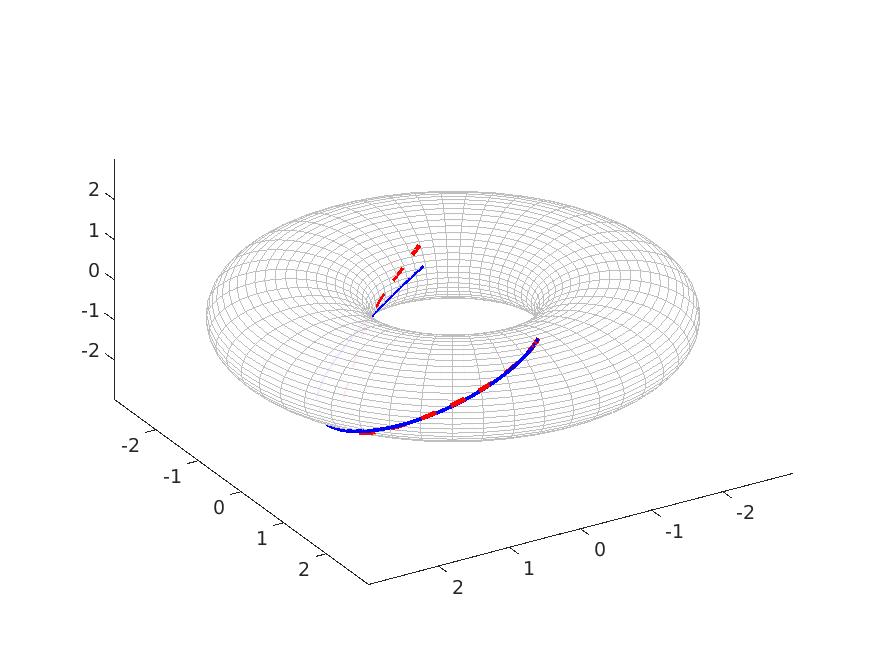}\hspace*{-8mm}
\includegraphics[width=.52\linewidth]{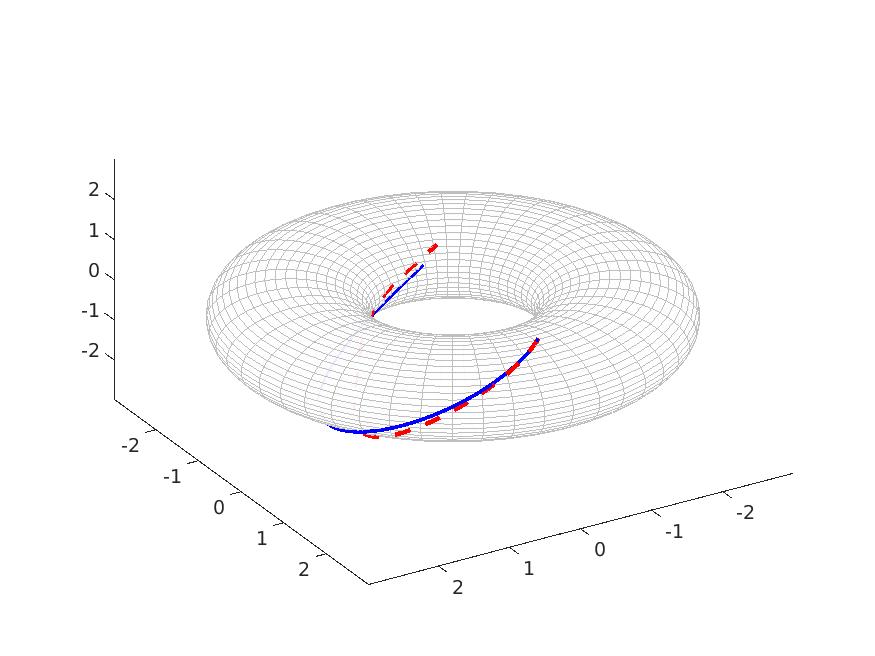} \\[-5.1mm]
\includegraphics[width=.52\linewidth]{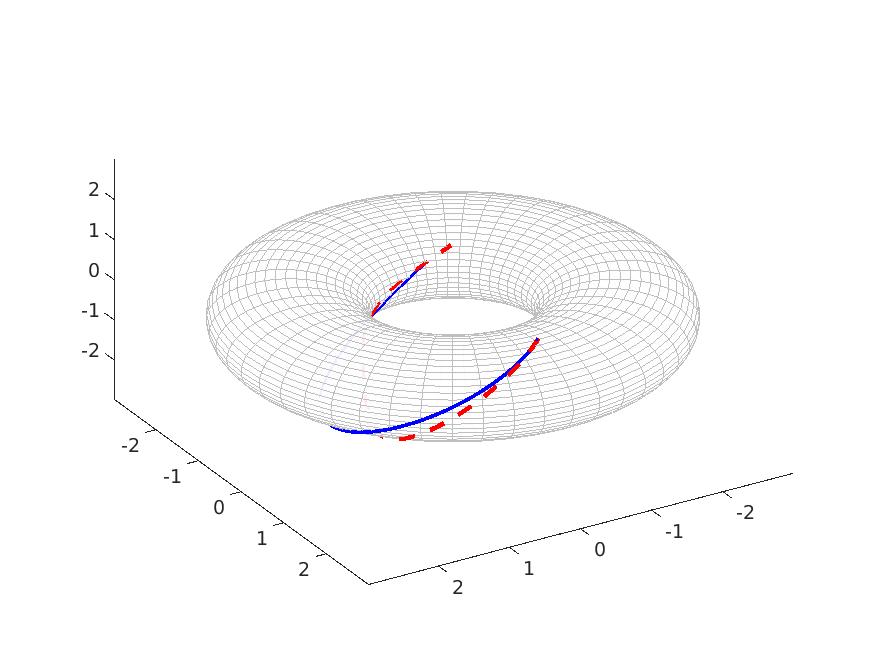}\hspace*{-8mm}
\includegraphics[width=.52\linewidth]{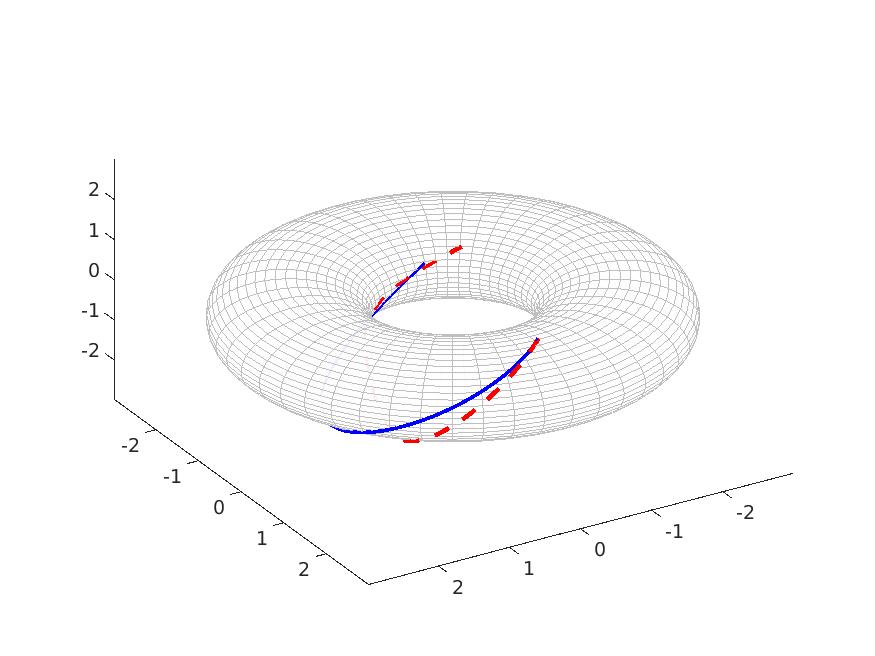} \\[-5.1mm]
\includegraphics[width=.52\linewidth]{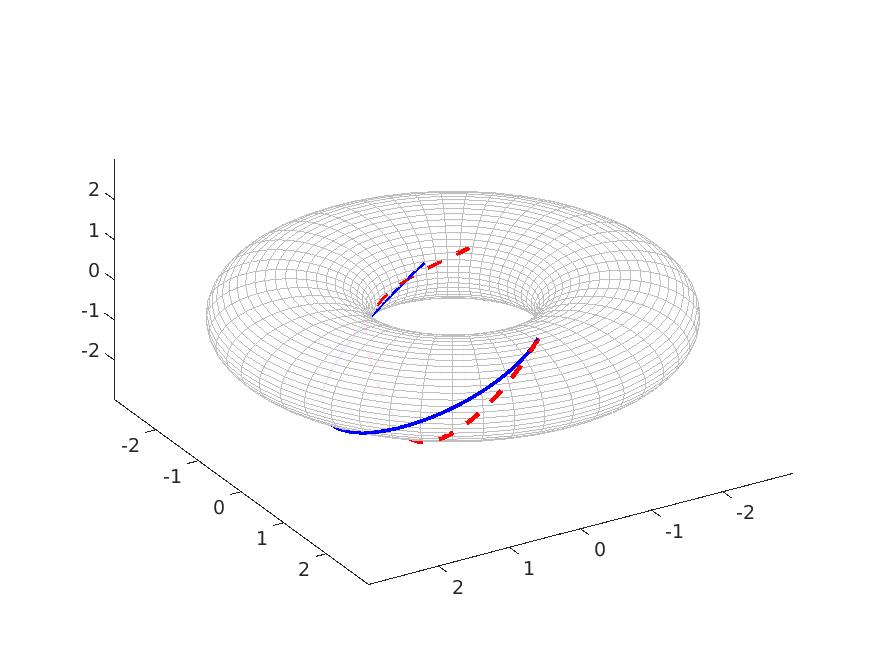}\hspace*{-8mm}
\includegraphics[width=.52\linewidth]{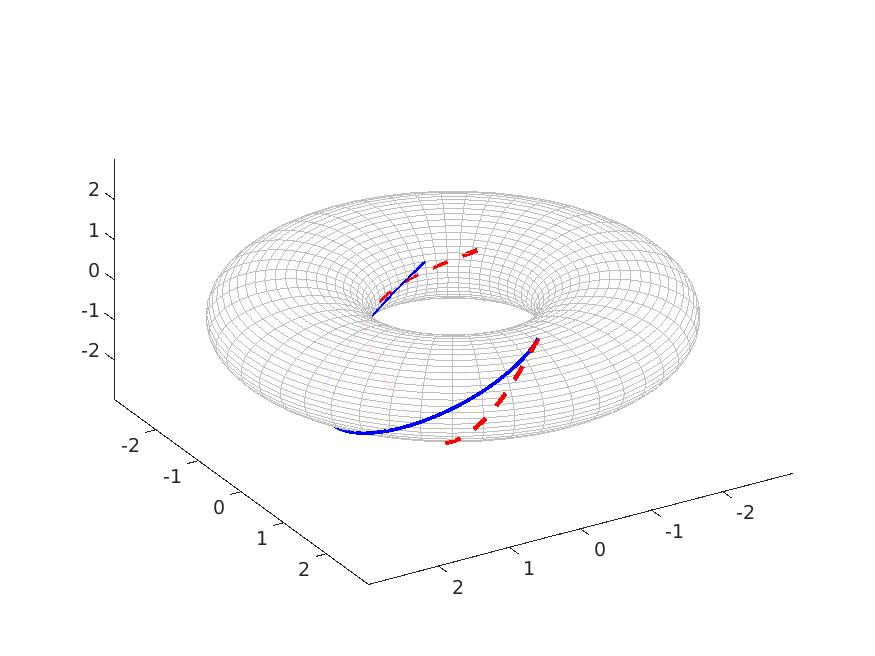} 
\caption{\label{fig:snaps_torus_comp} Snapshots of the discrete gradient flow
evolutions after $n=0,20,40,\dots,160$ iterations for the bending (solid curves) 
and geodesic curvature (dashed curves) energies from the same initial curve.}
\end{figure}

\begin{figure}[pt]
\includegraphics[width=.75\linewidth]{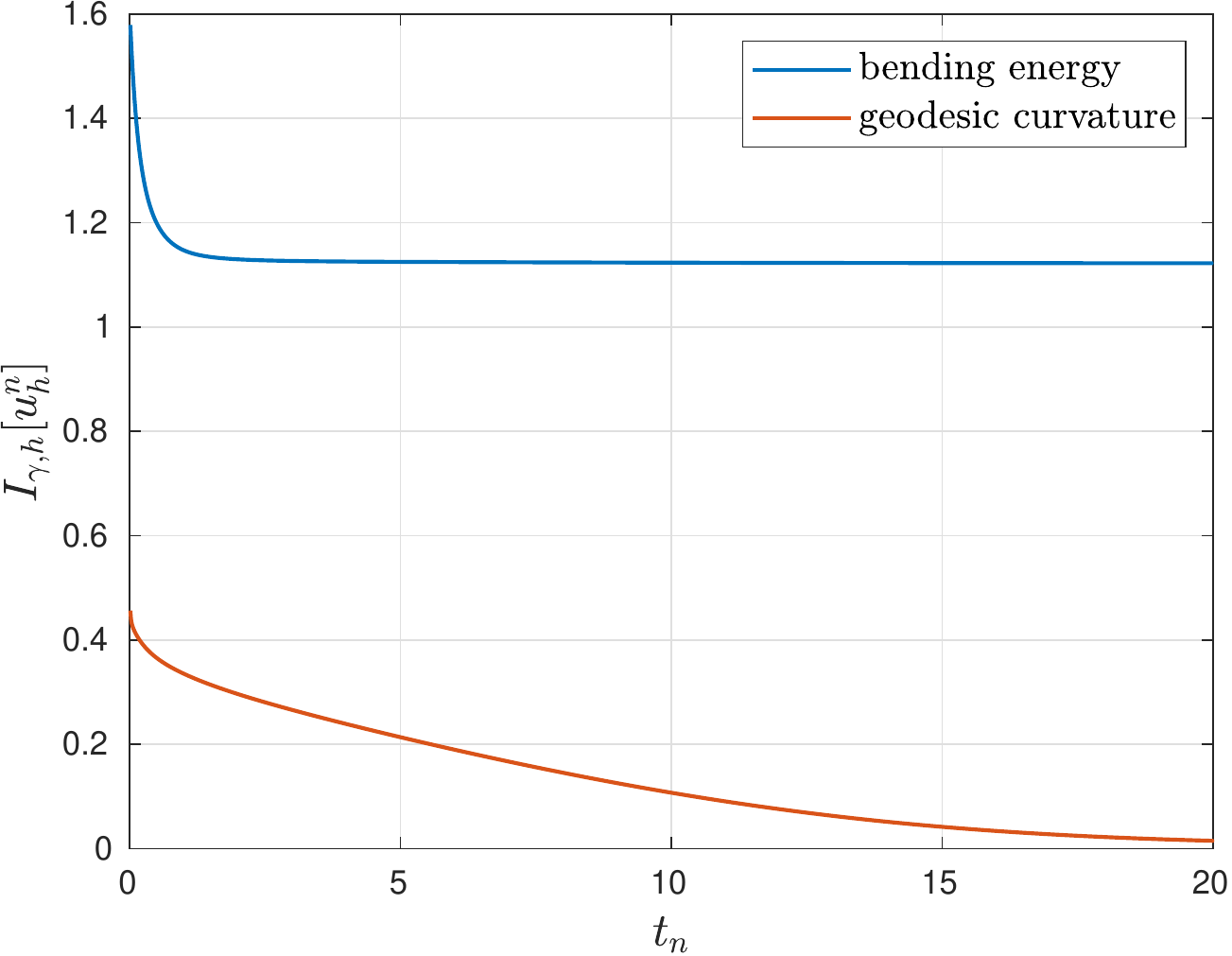}
\caption{\label{fig:torus_comp_eners} Decay of the bending energy
and geodesic curvature functional for the evolution of clamped 
curves on a torus illustrated in Figure~\ref{fig:snaps_torus_comp}.}
\end{figure}

\begin{figure}[pb]
\includegraphics[width=.8\linewidth]{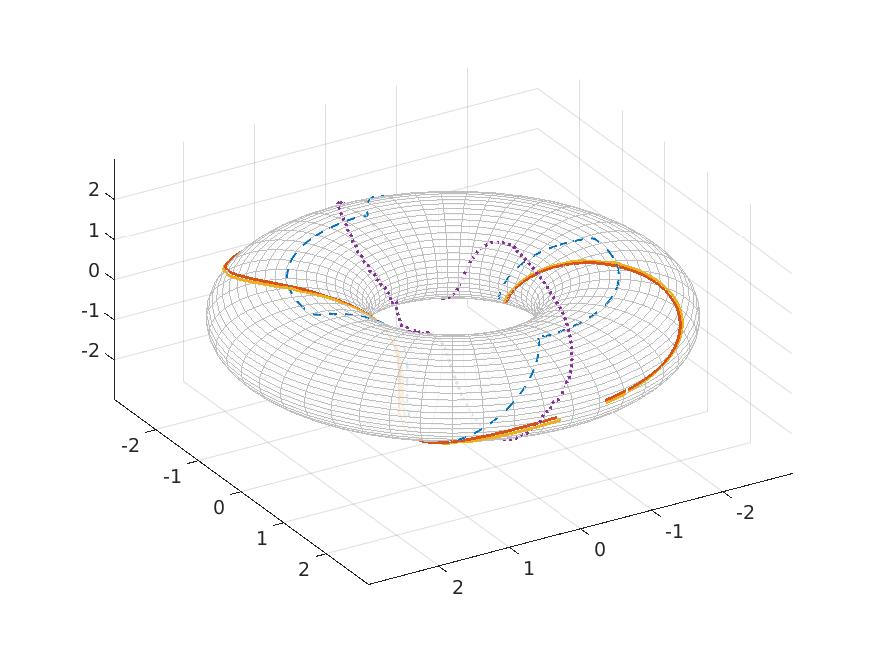}
\caption{\label{fig:torus_instability} Initial closed curve (dashed)
and corresponding relaxed curves for bending energy and geodesic
curvature with stabilization (nearly coinciding solid curves) and
configuration for geodesic curvature flow without stabilization 
(dotted irregular curve).}
\end{figure}

\subsection{Conical sheet indentation}
We consider the following specification of the conical
sheet indentation problem~\eqref{prob:indent}.

\begin{example}[Conical sheet]\label{ex:indent}
Let $\d = 1/4$ and $r= (15/16)^2$.
\end{example}

We used uniform partitions with mesh size $h>0$ and nodes 
$0=z_0<\dots < z_J=2\pi$
of the cirle $S^1$ where $z_0$ and $z_J$ are identified in the sense that
we impose the periodic boundary condition $L_\bc[u]=0$ with
\[
L_\bc[u] = \big(u(z_J)-u(z_0),u'(z_J)-u'(z_0)\big).
\]
Figure~\ref{fig:snaps_indent} shows snapshots of the discrete 
evolution computed with Algorithm~\ref{alg:conical_sheet} for
the discretization parameters
\[
J = 80, \quad h = 2\pi/J, \quad \veps = h^2, \quad \tau = h.
\]
The visualization displays the two-dimensional deformation of
the elastic sheet by linearly connecting the origin with points
on the curve. The initial configuration $u_h^0$ is a randomly 
generated function
with corrected values to satisfy the condition $u_h^0\in \cA_h$. 
The discrete evolution shows a rapid change to a smooth curve
approximately obeying the obstacle constraint. In the following iterations
the number of local maxima decreases until finally only one
fold can be observed while the remaining part of the curve
is in contact with the obstacle. Only a small penetration 
error occurs as can be seen in Figure~\ref{fig:indent_penetrate},
where we plotted the third component of the iterates $u_h^n$
with $n$ such that $t_n = n\tau = 2$, i.e., $n=160$,
 for the choices
\[
\textrm{(i)} \quad \veps = h, \quad
\textrm{(ii)} \quad \veps = h^2, \quad
\textrm{(iii)} \quad \veps = h^3.
\]
For $\veps = h$ we observe a strong penetration of the
obstacle. Our energy monotonicity property implies the estimate
\[
\|(u_{3,h}^n - \d)_-\| \le (2e_{0,h})^{1/2} \veps^{1/2}
\]
and from the experimental results we infer that $\veps =h^2$
leads to the best results. It is also interesting to see
how smaller penalization terms decrease the speed of the
relaxation process. For $\veps =h$ only one fold is present
indicating stationarity, while for $\veps = h^2$ and $\veps = h^3$
a larger number of local maxima can be observed after~160 iteration
steps. 
Figure~\ref{fig:indent_energy} shows the decay of the bending
energy for different resolutions and confirms the energy monotonicity
and convergence to a stationary configuration. The large values
of the energies are related to a strong dependence of minimal energies
on the indentation depth $\d$. For the significantly smaller
choice $\d= 0.05$ we obtained the stationary energy values
$I_{h,\veps}[u_h^{n^*}] = 24.104, 17.828, 21.822, 21.569, 21.565$ 
for discretizations with $J=40,80,\dots,640$ grid points. These
values also confirms convergence of the discrete minimal energies
as $h\to 0$. 

\begin{figure}[p]
\includegraphics[width=.47\linewidth]{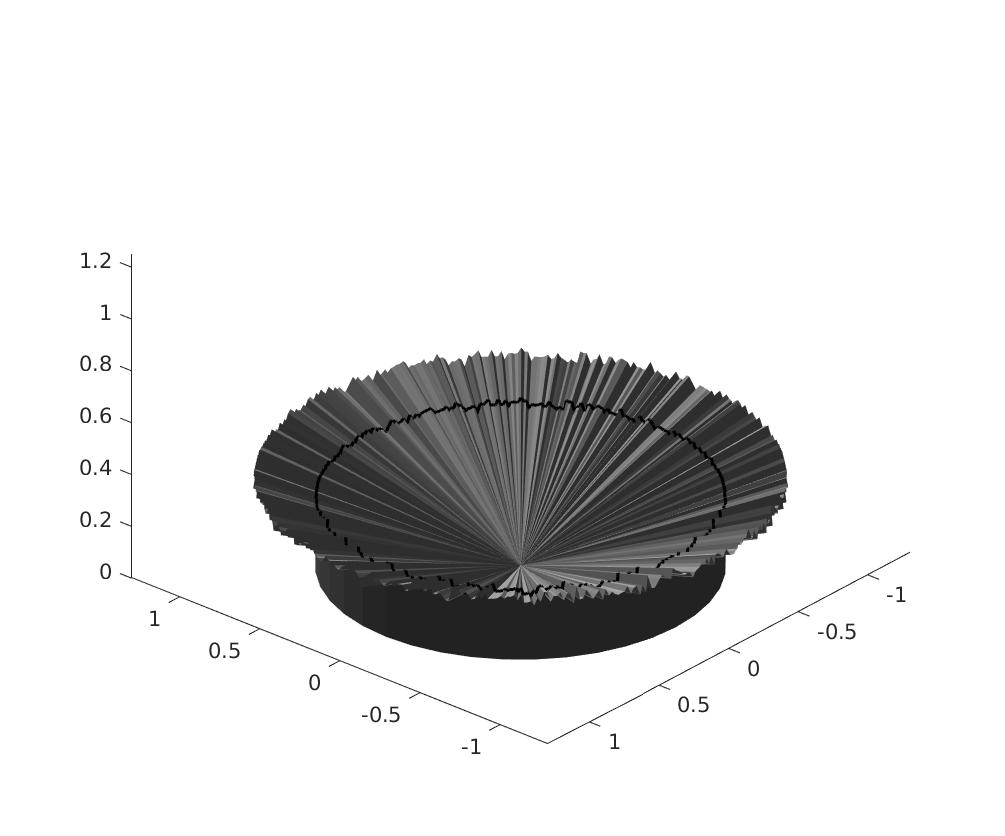}\hspace*{3mm}
\includegraphics[width=.47\linewidth]{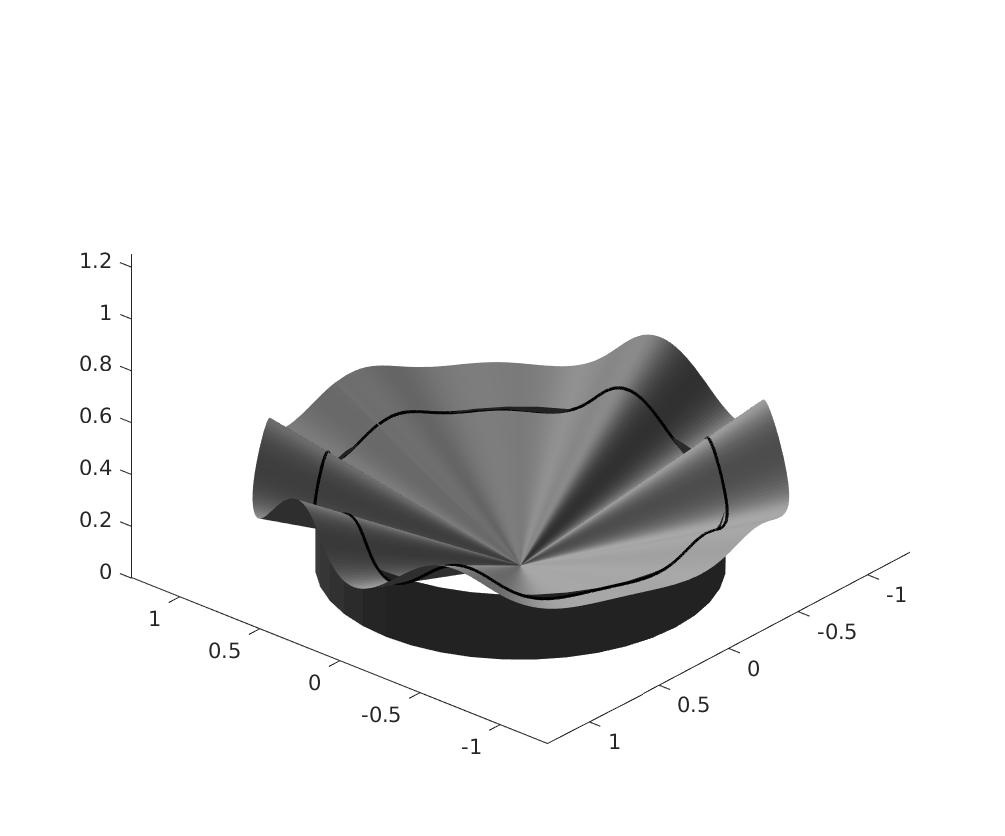} \\[-5.1mm]
\includegraphics[width=.47\linewidth]{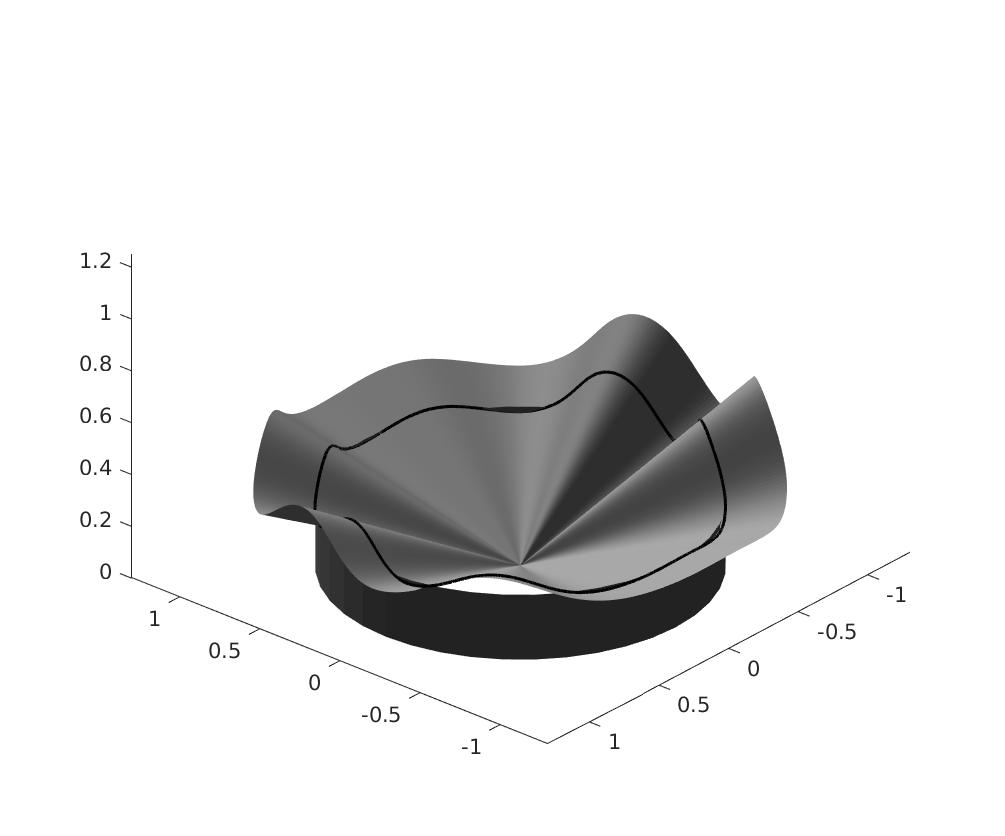}\hspace*{3mm}
\includegraphics[width=.47\linewidth]{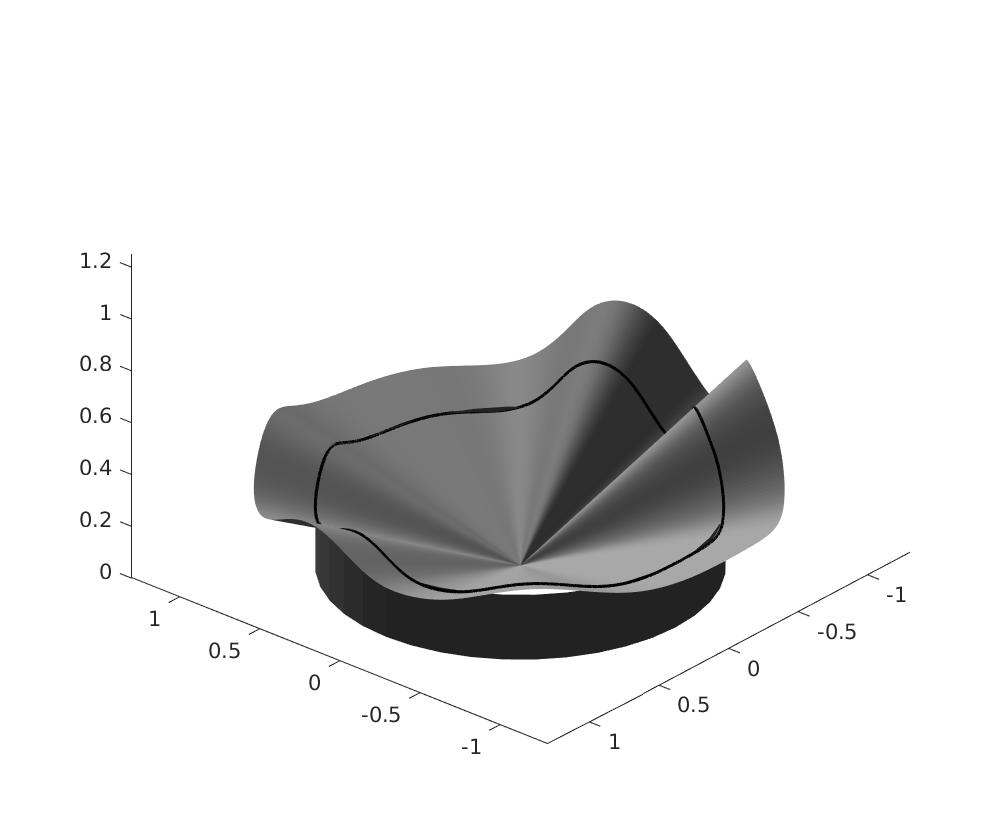} \\[-5.1mm]
\includegraphics[width=.47\linewidth]{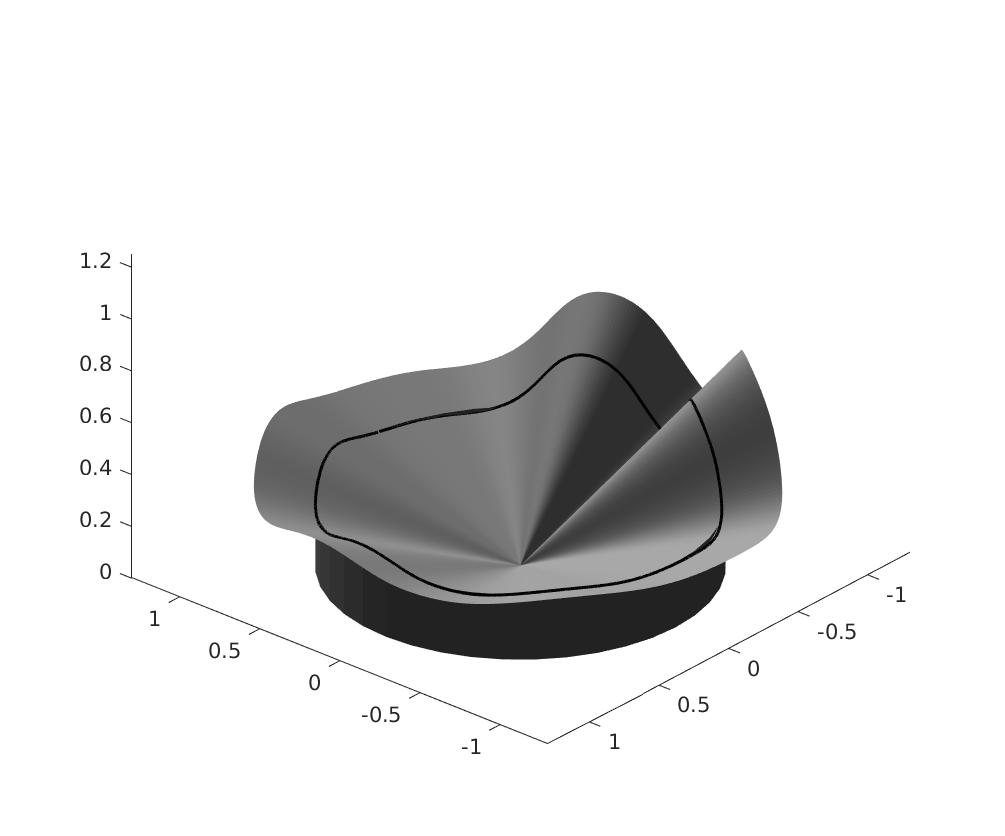}\hspace*{3mm}
\includegraphics[width=.47\linewidth]{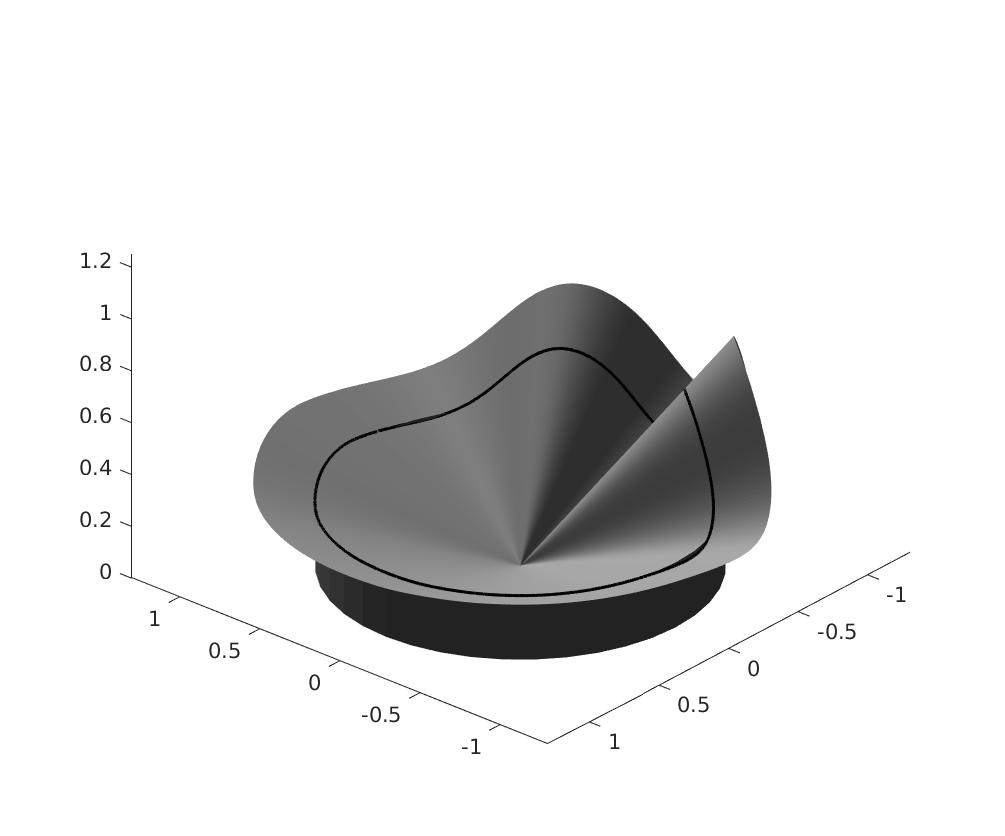} \\[-5.1mm]
\includegraphics[width=.47\linewidth]{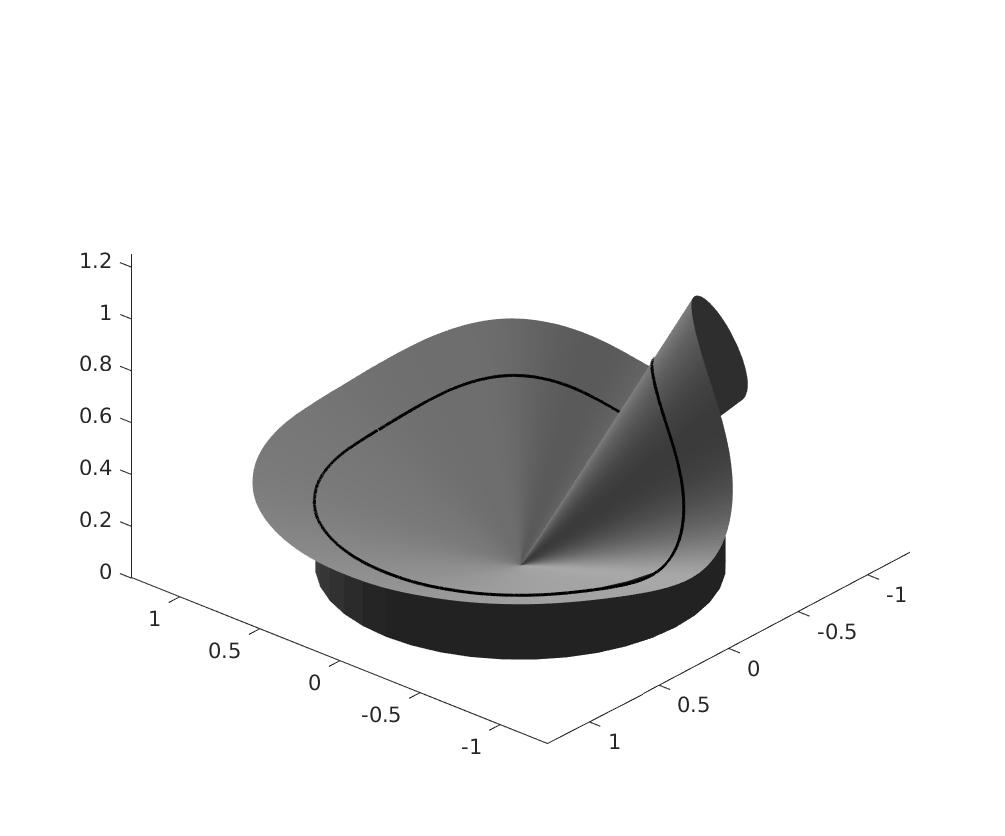}\hspace*{3mm}
\includegraphics[width=.47\linewidth]{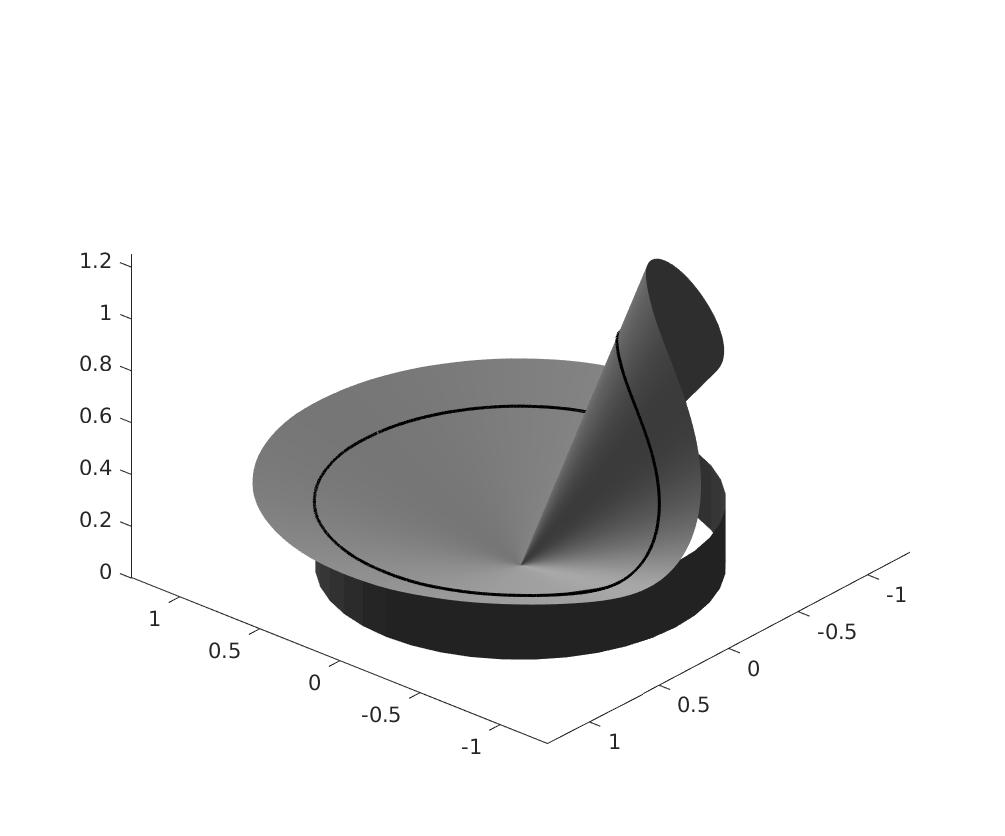} 
\caption{\label{fig:snaps_indent} Snapshots of the discrete gradient
flow for the sheet indentation problem after 
$n=0,10,20,30,40,70,190,430$ iterations. The family of 
curves (solid lines) relax their bending energies until only one fold 
is present which is stationary and energy minimizing.}
\end{figure}

\begin{figure}[pt]
\includegraphics[width=.75\linewidth]{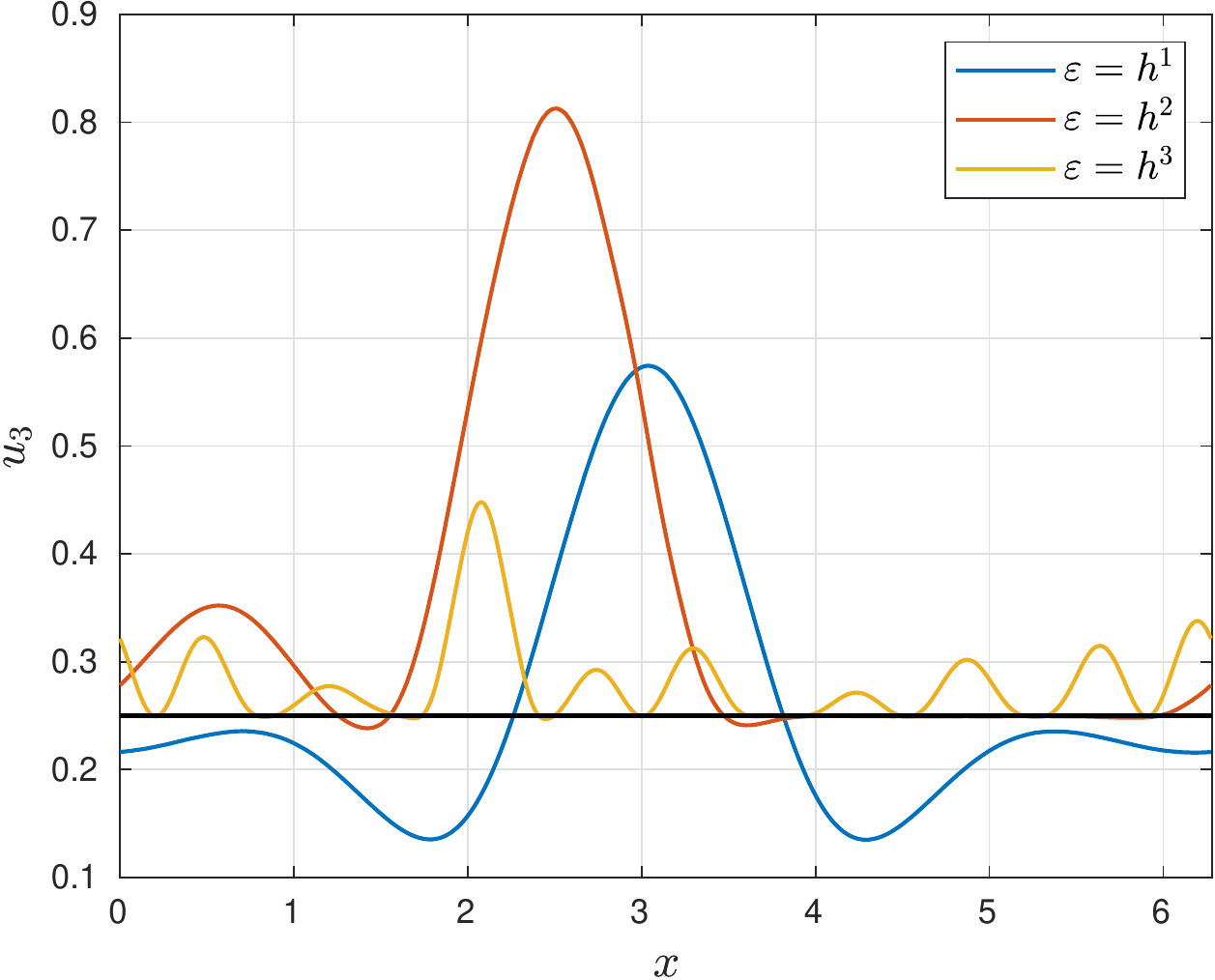}
\caption{\label{fig:indent_penetrate} Penetration of the obstacle
at height $\d = 0.25$ (straight line) by the third component $u_{3,h}^n$ 
of the iterates in the sheet indentation problem 
after a fixed number $n=160$ of iterations 
for different choices of penalty parameters $\veps$.}
\end{figure}

\begin{figure}[pb]
\includegraphics[width=.75\linewidth]{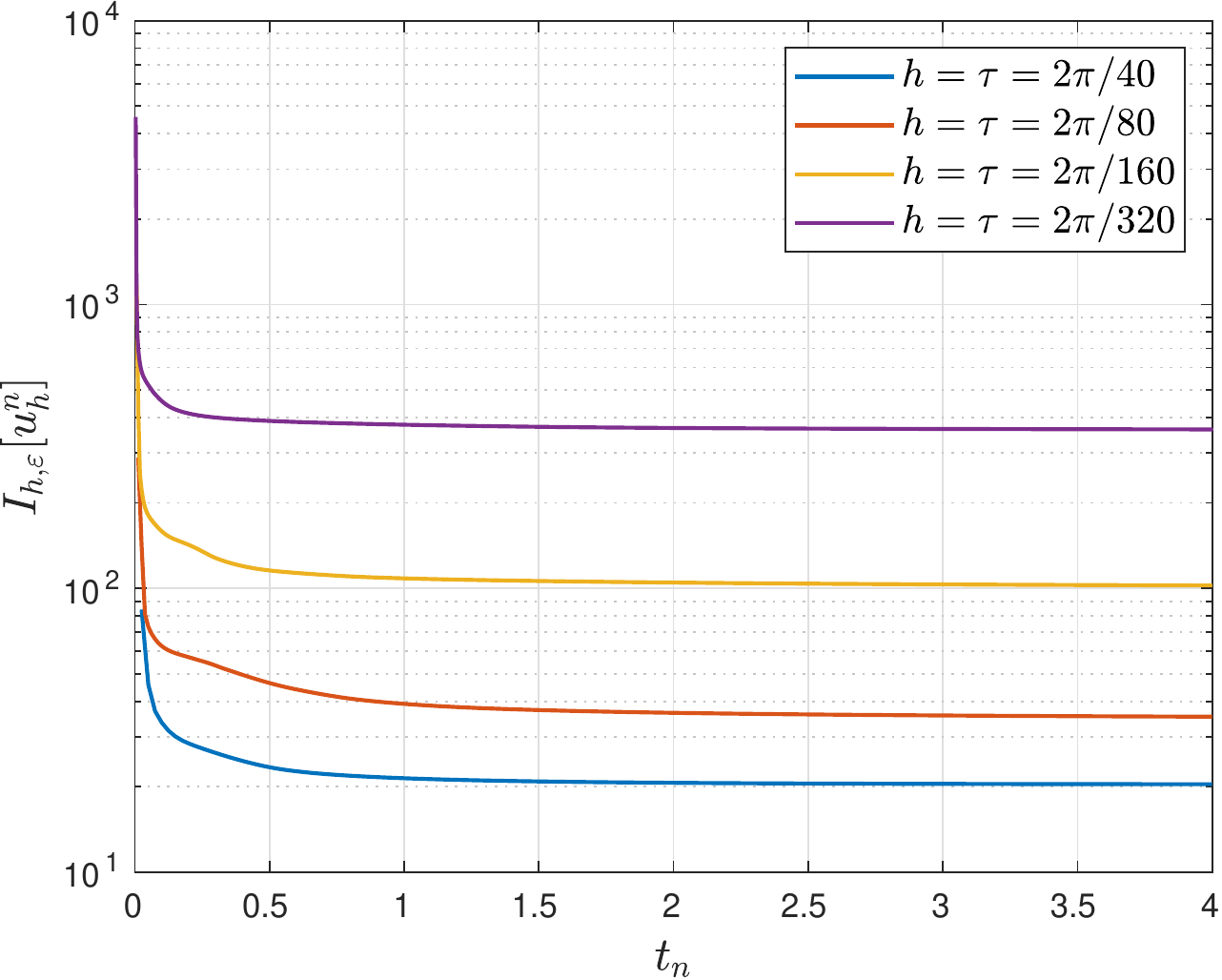}
\caption{\label{fig:indent_energy} Energy decay $n\mapsto I_{h,\veps}[u_h^n]$
in the 
sheet indentation problem for different spatial resolutions
for mesh-dependent randomly generated initial configurations
of large bending energy.}
\end{figure}

\medskip
\noindent
{\em Acknowledgments.} The author is grateful to Rebecca Kromer for 
providing first versions of the numerical experiments. 

\medskip

\section*{References}
\printbibliography[heading=none]

\end{document}